\documentclass[reqno,11pt]{amsart}
\usepackage{mathrsfs}
\usepackage{verbatim}
\usepackage{upgreek}
\usepackage{amsmath}
\usepackage{amsxtra}
\usepackage{amscd}
\usepackage{amsthm}
\usepackage{amsfonts}
\usepackage{amssymb}
\usepackage{eucal}
\usepackage{tikz-cd}
\usepackage[matrix,arrow,curve]{xy}
\usepackage{young}
\usepackage[vcentermath]{youngtab}
\usepackage{hyperref}

\DeclareMathOperator{\sph}{sph}

\newcommand{\rParB}[1]{\left[#1\right]}

\theoremstyle{plain}
\newtheorem{Thm}{Theorem}[section]
\newtheorem{Cor}[Thm]{Corollary}
\newtheorem{Lem}[Thm]{Lemma}
\newtheorem{Prop}[Thm]{Proposition}
\newtheorem{Conj}[Thm]{Conjecture}

\theoremstyle{definition}
\newtheorem{Def}[Thm]{Definition}

\theoremstyle{definition}
\newtheorem{Ex}[Thm]{Example}

\theoremstyle{remark}
\newtheorem{Rem}[Thm]{Remark}

\newtheorem{conj}{Conjecture}

\numberwithin{equation}{section}
\renewcommand{\rm}{\normalshape}

\newif\ifShowLabels
\ShowLabelstrue
\newdimen\theight
\def\TeXref#1{%
	\leavevmode\vadjust{\setbox0=\hbox{{\tt
				\quad\quad  {\small \rm #1}}}%
		\theight=\ht0
		\advance\theight by \lineskip
		\kern -\theight \vbox to
		\theight{\rightline{\rlap{\box0}}%
			\vss}%
}}%

\ShowLabelsfalse

\renewcommand{\sec}[2]{\section{#2}\label{S:#1}%
	\ifShowLabels \TeXref{{S:#1}} \fi}
\newcommand{\ssec}[2]{\subsection{#2}\label{SS:#1}%
	\ifShowLabels \TeXref{{SS:#1}} \fi}

\newcommand{\sssec}[2]{\subsubsection{#2}\label{SSS:#1}%
	\ifShowLabels \TeXref{{SSS:#1}} \fi}

\newenvironment{thm}[1]%
{ \begin{Thm} \label{T:#1}  \ifShowLabels \TeXref{T:#1} \fi }%
	{ \end{Thm} }

\renewcommand{\th}[1]{\begin{thm}{#1} \sl }
	\renewcommand{\eth}{\end{thm} }

\newenvironment{lemma}[1]%
{ \begin{Lem} \label{L:#1}  \ifShowLabels \TeXref{L:#1} \fi }%
	{ \end{Lem} }
\newcommand{\lem}[1]{\begin{lemma}{#1} \sl}
	\newcommand{\elem}{\end{lemma}}

\newenvironment{propos}[1]%
{ \begin{Prop} \label{P:#1}  \ifShowLabels \TeXref{P:#1} \fi }%
	{ \end{Prop} }
\newcommand{\prop}[1]{\begin{propos}{#1}\sl }
	\newcommand{\eprop}{\end{propos}}

\newenvironment{corol}[1]%
{ \begin{Cor} \label{C:#1}  \ifShowLabels \TeXref{C:#1} \fi }%
	{ \end{Cor} }
\newcommand{\cor}[1]{\begin{corol}{#1} \sl }
	\newcommand{\ecor}{\end{corol}}

\newenvironment{defeni}[1]%
{ \begin{Def} \label{D:#1}  \ifShowLabels \TeXref{D:#1} \fi }%
	{ \end{Def} }
\newcommand{\defe}[1]{\begin{defeni}{#1} \sl }
	\newcommand{\edefe}{\end{defeni}}

\newenvironment{remark}[1]%
{ \begin{Rem} \label{R:#1}  \ifShowLabels \TeXref{R:#1} \fi }%
	{ \end{Rem} }
\newcommand{\rem}[1]{\begin{remark}{#1}}
	\newcommand{\erem}{\end{remark}}

\newenvironment{conjec}[1]%
{ \begin{Conj} \label{Co:#1}  \ifShowLabels \TeXref{Co:#1} \fi }%
	{ \end{Conj} }
\renewcommand{\conj}[1]{\begin{conjec}{#1} \sl }
	\newcommand{\econj}{\end{conjec}}

\newenvironment{First proof}[1]%
{ \begin{First proof} \label{Co:#1}  \ifShowLabels \TeXref{Co:#1} \fi }%
	{ \end{First proof} }

\newenvironment{Second proof}[1]%
{ \begin{Second proof} \label{Co:#1}  \ifShowLabels \TeXref{Co:#1} \fi }%
	{ \end{Second proof} }

\newcommand{\eq}[1]%
{ \ifShowLabels \TeXref{E:#1} \fi
	\begin{equation} \label{E:#1} }
\newcommand{\eeq}{ \end{equation} }

\newcommand{\prf}{ \begin{proof} }
	\newcommand{\epr}{ \end{proof} }

\newcommand{\prft}{ \begin{proof} }
	\newcommand{\eprt}{ \end{proof} }

\newcommand\nc{\newcommand}

\nc{\HC}{{\mathcal{HC}}}
\nc{\on}{\operatorname}
\nc{\BA}{{\mathbb{A}}}
\nc{\BC}{{\mathbb{C}}}
\nc{\BF}{{\mathbb{F}}}
\nc{\BG}{{\mathbb{G}}}
\nc{\BM}{{\mathbb{M}}}
\nc{\BN}{{\mathbb{N}}}
\nc{\BO}{{\mathbb{O}}}
\nc{\BQ}{{\mathbb{Q}}}
\nc{\BP}{{\mathbb{P}}}
\nc{\BR}{{\mathbb{R}}}
\nc{\BZ}{{\mathbb{Z}}}
\nc{\BS}{{\mathbb{S}}}

\nc{\CA}{{\mathcal{A}}}
\nc{\CB}{{\mathcal{B}}}
\nc{\CalC}{{\mathcal C}}
\nc{\CalD}{{\mathcal D}}
\nc{\CE}{{\mathcal{E}}}
\nc{\CF}{{\mathcal{F}}}
\nc{\CG}{{\mathcal{G}}}
\nc{\CH}{{\mathcal{H}}}
\nc{\CK}{{\mathcal{K}}}
\nc{\CL}{{\mathcal{L}}}
\nc{\CM}{{\mathcal{M}}}
\nc{\CMM}{{\mathcal{M}^{\operatorname{gen}}_\hbar(-\rho)}}
\nc{\CN}{{\mathcal{N}}}
\nc{\CO}{{\mathcal{O}}}
\nc{\CP}{{\mathcal{P}}}
\nc{\CQ}{{\mathcal{Q}}}
\nc{\CR}{{\mathcal{R}}}
\nc{\CS}{{\mathcal{S}}}
\nc{\CT}{{\mathcal{T}}}
\nc{\CU}{{\mathcal{U}}}
\nc{\CV}{{\mathcal{V}}}
\nc{\CW}{{\mathcal{W}}}
\nc{\CX}{{\mathcal{X}}}
\nc{\CY}{{\mathcal{Y}}}
\nc{\CZ}{{\mathcal{Z}}}

\nc{\gen}{{\operatorname{gen}}}
\nc{\cM}{{\check{\mathcal M}}{}}
\nc{\csM}{{\check{\mathcal A}}{}}
\nc{\obM}{{\overset{\circ}{\mathbf M}}{}}
\nc{\oCA}{{\overset{\circ}{\mathcal A}}{}}
\nc{\obA}{{\overset{\circ}{\mathbf A}}{}}
\nc{\ooM}{{\overset{\circ}{M}}{}}
\nc{\osM}{{\overset{\circ}{\mathsf M}}{}}
\nc{\vM}{{\overset{\bullet}{\mathcal M}}{}}
\nc{\nM}{{\underset{\bullet}{\mathcal M}}{}}
\nc{\obD}{{\overset{\circ}{\mathbf D}}{}}
\nc{\cp}{{\overset{\circ}{\mathbf p}}{}}
\nc{\ofZ}{{\overset{\circ}{\mathfrak Z}}{}}

\nc{\fa}{{\mathfrak{a}}}
\nc{\fb}{{\mathfrak{b}}}
\nc{\fg}{{\mathfrak{g}}}
\nc{\fgl}{{\mathfrak{gl}}}
\nc{\fh}{{\mathfrak{h}}}
\nc{\fj}{{\mathfrak{j}}}
\nc{\fm}{{\mathfrak{m}}}
\nc{\fn}{{\mathfrak{n}}}
\nc{\fu}{{\mathfrak{u}}}
\nc{\fp}{{\mathfrak{p}}}
\nc{\frr}{{\mathfrak{r}}}
\nc{\fs}{{\mathfrak{s}}}
\nc{\ft}{{\mathfrak{t}}}
\nc{\fT}{{\mathfrak{T}}}
\nc{\ofT}{{\overline{\mathfrak T}}}
\nc{\ofS}{{\overline{\mathfrak S}}}
\nc{\fsl}{{\mathfrak{sl}}}
\nc{\hsl}{{\widehat{\mathfrak{sl}}}}
\nc{\hgl}{{\widehat{\mathfrak{gl}}}}
\nc{\hg}{{\widehat{\mathfrak{g}}}}
\nc{\chg}{{\widehat{\mathfrak{g}}}{}^\vee}
\nc{\hn}{{\widehat{\mathfrak{n}}}}
\nc{\chn}{{\widehat{\mathfrak{n}}}{}^\vee}

\nc{\fA}{{\mathfrak{A}}}
\nc{\fB}{{\mathfrak{B}}}
\nc{\fD}{{\mathfrak{D}}}
\nc{\fE}{{\mathfrak{E}}}
\nc{\fF}{{\mathfrak{F}}}
\nc{\fG}{{\mathfrak{G}}}
\nc{\fI}{{\mathfrak{I}}}
\nc{\fJ}{{\mathfrak{J}}}
\nc{\fK}{{\mathfrak{K}}}
\nc{\fL}{{\mathfrak{L}}}
\nc{\fM}{{\mathfrak{M}}}
\nc{\fN}{{\mathfrak{N}}}
\nc{\frP}{{\mathfrak{P}}}
\nc{\fQ}{{\mathfrak Q}}
\nc{\fS}{{\mathfrak S}}
\nc{\fU}{{\mathfrak{U}}}
\nc{\fZ}{{\mathfrak{Z}}}

\nc{\ba}{{\mathbf{a}}}
\nc{\bb}{{\mathbf{b}}}
\nc{\bc}{{\mathbf{c}}}
\nc{\bd}{{\mathbf{d}}}
\nc{\be}{{\mathbf{e}}}
\nc{\bi}{{\mathbf{i}}}
\nc{\bj}{{\mathbf{j}}}
\nc{\bn}{{\mathbf{n}}}
\nc{\bp}{{\mathbf{p}}}
\nc{\br}{{\mathbf{r}}}
\nc{\bq}{{\mathbf{q}}}
\nc{\bu}{{\mathbf{u}}}
\nc{\bv}{{\mathbf{v}}}
\nc{\bx}{{\mathbf{x}}}
\nc{\by}{{\mathbf{y}}}
\nc{\bw}{{\mathbf{w}}}
\nc{\bA}{{\mathbf{A}}}
\nc{\bB}{{\mathbf{B}}}
\nc{\bC}{{\mathbf{C}}}
\nc{\bD}{{\mathbf{D}}}
\nc{\bE}{{\mathbf{E}}}
\nc{\bK}{{\mathbf{K}}}
\nc{\bH}{{\mathbf{H}}}
\nc{\bM}{{\mathbf{M}}}
\nc{\bN}{{\mathbf{N}}}
\nc{\bO}{{\mathbf{O}}}
\nc{\bQ}{{\mathbf Q}}
\nc{\bS}{{\mathbf{S}}}
\nc{\bT}{{\mathbf{T}}}
\nc{\bV}{{\mathbf{V}}}
\nc{\bW}{{\mathbf{W}}}
\nc{\bX}{{\mathbf{X}}}
\nc{\bP}{{\mathbf{P}}}
\nc{\bZ}{{\mathbf{Z}}}

\nc{\sA}{{\mathsf{A}}}
\nc{\sB}{{\mathsf{B}}}
\nc{\sC}{{\mathsf{C}}}
\nc{\sD}{{\mathsf{D}}}
\nc{\sF}{{\mathsf{F}}}
\nc{\sK}{{\mathsf{K}}}
\nc{\sM}{{\mathsf{M}}}
\nc{\sO}{{\mathsf{O}}}
\nc{\sQ}{{\mathsf{Q}}}
\nc{\sP}{{\mathsf{P}}}
\nc{\sV}{{\mathsf{V}}}
\nc{\sW}{{\mathsf{W}}}
\nc{\sZ}{{\mathsf{Z}}}
\nc{\sfp}{{\mathsf{p}}}
\nc{\sr}{{\mathsf{r}}}
\nc{\st}{{\mathsf{t}}}
\nc{\sfb}{{\mathsf{b}}}
\nc{\sfc}{{\mathsf{c}}}
\nc{\sd}{{\mathsf{d}}}
\nc{\sg}{{\mathsf{g}}}
\nc{\sk}{{\mathsf{k}}}
\nc{\sfl}{{\mathsf{l}}}

\nc{\BK}{{\overline{K}}}

\nc{\tA}{{\widetilde{\mathbf{A}}}}
\nc{\tB}{{\widetilde{\mathcal{B}}}}
\nc{\tg}{{\widetilde{\mathfrak{g}}}}
\nc{\tG}{{\widetilde{G}}}
\nc{\TM}{{\widetilde{\mathbb{M}}}{}}
\nc{\tN}{{\widetilde{\mathcal{N}}}{}}
\nc{\tO}{{\widetilde{\mathsf{O}}}{}}
\nc{\tU}{{\widetilde{\mathfrak{U}}}{}}
\nc{\TZ}{{\tilde{Z}}}
\nc{\tZ}{\widetilde{Z}{}}
\nc{\tx}{{\tilde{x}}}
\nc{\tbv}{{\tilde{\bv}}}
\nc{\tfP}{{\widetilde{\mathfrak{P}}}{}}
\nc{\tz}{{\tilde{\zeta}}}
\nc{\tmu}{{\tilde{\mu}}}

\nc{\td}{\ddot{\underline{d}}{}}
\nc{\tzeta}{\widetilde{\zeta}{}}
\nc{\hd}{{\widehat{\underline{d}}}}
\nc{\hG}{{\widehat{G}}}
\nc{\hBP}{\widehat{\mathbb P}{}}
\nc{\hQ}{{\widehat{Q}}}
\nc{\hsM}{\widehat{\mathsf M}{}}
\nc{\hfM}{\widehat{\mathfrak M}{}}
\nc{\hCP}{\widehat{\mathcal P}{}}
\nc{\hCR}{\widehat{\mathcal R}{}}
\nc{\hCS}{{\widehat{\mathcal S}}}
\nc{\hfZ}{\widehat{\mathfrak Z}{}}
\nc{\hZ}{\widehat{Z}{}}

\nc{\urho}{\underline{\rho}}
\nc{\uB}{\underline{B}}
\nc{\uC}{{\underline{\mathbb{C}}}}
\nc{\ui}{\underline{i}}
\nc{\ofP}{{\overline{\mathfrak{P}}}}

\nc{\hrho}{{\hat{\rho}}}

\nc{\unl}{\underline}
\nc{\ol}{\overline}
\nc{\one}{{\mathbf{1}}}
\nc{\two}{{\mathbf{t}}}

\nc{\Sym}{{\mathop{\operatorname{Sym}}}}
\nc{\Tot}{{\mathop{\operatorname{\normalshape Tot}}}}
\nc{\Hilb}{{\mathop{\operatorname{\normalshape Hilb}}}}
\nc{\Hom}{{\mathop{\operatorname{Hom}}}}
\nc{\CHom}{{\mathop{\operatorname{{\mathcal{H}}\it om}}}}
\nc{\defi}{{\mathop{\operatorname{\normalshape def}}}}
\nc{\length}{{\mathop{\operatorname{\normalshape length}}}}

\nc{\Cliff}{{\mathsf{Cliff}}}
\nc{\Fl}{{\mathcal{F}\ell}}
\nc{\Fib}{{\mathsf{Fib}}}
\nc{\Coh}{{\mathsf{Coh}}}
\nc{\FCoh}{{\mathsf{FCoh}}}

\nc{\reg}{{\text{\normalshape reg}}}
\nc{\res}{{\operatorname{res}}}
\nc{\cplus}{{\mathbf{C}_+}}
\nc{\cminus}{{\mathbf{C}_-}}
\nc{\cthree}{{\mathbf{C}_*}}
\nc{\Qbar}{{\overline{Q}}}

\nc{\bh}{{\overline{h}}}
\nc{\bOmega}{{\overline{\Omega}}}
\nc\tGr{\widetilde{\Gr}}

\nc{\seq}[1]{\stackrel{#1}{\sim}}
\nc\ogu{\overline{G/U}}
\nc\chlam{\check{\lam}}

\nc\St{\operatorname{St}}
\nc{\oZ}{{\overset{\circ}{Z}}}
\nc{\tF}{\widetilde{\mathcal F}}

\nc\uS{\underline{S}}
\nc\QM{\mathcal{QM}}

\nc{\chmu}{\check{\mu}}
\newcommand\iso{\,\vphantom{j^{X^2}}\smash{\overset{\sim}{\vphantom{\rule{0pt}{0.20em}}\smash{\longrightarrow}}}\,}
\nc{\ul}{\underline}
\nc{\Mvd}{\mathfrak{M}(\underline{v},\underline{d})}
\nc{\MvdT}{\mathfrak{M}(\underline{v}^{\dagger},\underline{d}^{\dagger})}
\nc{\MVD}{\mathfrak{M}(V,D)}
\nc{\mt}{\mapsto}
\nc{\sm}{\setminus}
\nc{\ra}{\rightarrow}
\nc{\lar}{\leftarrow}
\nc{\hr}{\hookrightarrow}
\nc{\La}{\Lambda}
\nc{\Lap}{\Lambda^{+}}
\nc{\oZal}{\overset{\circ}{Z^{\alpha}}}
\nc{\sig}{\sigma}
\nc{\al}{\alpha}
\nc{\la}{\lambda}
\nc{\is}{\simeq}
\nc{\ip}{\iota^{+}_{\la, \mu}}
\nc{\im}{\iota^{-}_{\la, \mu}}
\nc{\jp}{j^{+}_{\la, \mu}}
\nc{\jm}{j^{-}_{\la, \mu}}
\nc{\pip}{\pi^{+}_{\la, \mu}}
\nc{\pim}{\pi^{-}_{\la, \mu}}
\nc{\s}{\star}
\nc{\fpt}{[A^{\la},B^{\la},\gamma^{\la},\delta^{\la}]}
\nc{\ulfpt}{[\ul{A}^{\la},\ul{B}^{\la},\ul{\gamma}^{\la},\ul{\delta}^{\la}]}
\nc{\lvee}{\!\scriptscriptstyle\vee}
\nc{\Gr}{{\operatorname{Gr}}}
\nc{\rra}{\twoheadrightarrow}

\nc{\End}{\on{End}}

\begin{document}
\author{Pavel Etingof}
\address{Department of Mathematics
Massachusetts Institute of Technology
\newline
77 Massachusetts Avenue,
Cambridge, MA 02139,
USA}
\email{etingof@math.mit.edu}
\author{Vasily Krylov}
\address{
Department of Mathematics
Massachusetts Institute of Technology
\newline
77 Massachusetts Avenue,
Cambridge, MA 02139,
USA;
\newline
Department of Mathematics National Research University Higher School of Economics, \newline 
6 Usacheva st., Moscow, 119048, Russian Federation
}
\email{krvas@mit.edu}	
\author{Ivan Losev}
\address{Department of Mathematics, Yale University 
\newline 12 Hillhouse Avenue, New Haven, CT 06511, USA}
\email{ivan.loseu@yale.edu}
\author{Jos\'e Simental}
\address{Department of Mathematics, University of California Davis \newline  One Shields Avenue, Davis, CA 95616, USA}
\email{jsimental@ucdavis.edu}
	
	\title[Representations of quantized Gieseker varieties]
	{Representations with minimal support  for quantized Gieseker varieties}
	
	\dedicatory{To the memory of Tom Nevins} 
	
	\begin{abstract} We study the minimally supported representations of quantizations of Gieseker moduli spaces. We relate them to $\operatorname{SL}_n$-equivariant D-modules on the nilpotent cone of $\mathfrak{sl}_n$ and to minimally supported representations of type A rational Cherednik algebras. Our main result is character formulas for minimally supported representations of quantized Gieseker moduli spaces. 
	\end{abstract}

	\maketitle
	\section{Introduction}
In this paper we continue the study of the representation theory of quantizations of Gieseker varieties started in~\cite{L}. The main focus of the paper is on the representations with minimal support. We will elaborate on what this means later in this section. We will obtain character formulas for these representations.

\ssec{Not}{Gieseker varieties}	
Pick two vector spaces $V,\,W$ of dimensions $n,\,r \in \BZ_{\geqslant 1}$ respectively. Consider the space $R:=\mathfrak{gl}(V) \oplus \on{Hom}(V,W)$ and a natural action of $\on{GL}(V)$ on it, let $\mathfrak{gl}(V) \ra TR,\, \xi \mapsto \xi_R$ be the infinitesimal action. We can form the cotangent bundle $T^*R$, this is a symplectic vector space. Identifying $\mathfrak{gl}(V)^*$ with $\mathfrak{gl}(V)$ and $\on{Hom}(V,W)^*$ with $\on{Hom}(W,V)$ by means of the trace form we identify $T^*R$ with $\mathfrak{gl}(V)^{\oplus 2} \oplus \on{Hom}(V,W)\oplus \on{Hom}(W,V)$.
	The action of $\on{GL}(V)$ on $T^*R$ is symplectic so we get the moment map $\mu\colon T^*R \ra \mathfrak{gl}(V)$. It can be described in two equivalent ways. First, we have $\mu(A,B,i,j)=[A,B]-ji$. Second, the dual map $\mu^*\colon \mathfrak{gl}(V) \ra \BC[T^*R]$ sends $\xi \in \mathfrak{gl}(V)$ to the vector field $\xi_R$ considered as a polynomial function on $T^*R$. 
	
	We identify the character lattice of $\on{GL}(V)$ with $\BZ$ via the map 
	$$
	\BZ \ni \theta \mapsto (\on{det}^{\theta}\colon \on{GL(V)} \ra \BC^{\times}).
	$$
	Let us pick $\theta \in \BZ \setminus \{0\}$ and consider the open set of $\theta$-stable points $(T^*R)^{\theta-\on{st}} \subset T^*R$. For $\theta>0$ the subset of stable points consists of all quadruples $(A,B,i,j)$ such that $\on{ker}i$ does not contain nonzero $A$- and $B$-stable subspaces. For $\theta<0$ the subset of stable points consists of all quadruples $(A,B,i,j)$ such that $V$ is the unique $A$- and $B$-stable subspace of $V$ containing $\on{im}j$.
	We can form the $\on{GIT}$ Hamiltonian reduction $\mathfrak{M}^\theta(n,r)=T^*R/\!\!/\!\!/^\theta \on{GL}(V):=\mu^{-1}(0)^{\theta-\on{st}}/\on{GL}(V)$. This is a smooth symplectic quasi-projective variety of dimension $2rn$ that is a resolution of singularities of the categorical Hamiltonian reduction $\mathfrak{M}(n,r):=\mu^{-1}(0)/\!\!/\on{GL}(V)$ which is a Poisson variety. 
	We note that $\mathfrak{M}^\theta(n,r)$ and $\mathfrak{M}^{-\theta}(n,r)$ are symplectomorphic via the isomorphism induced by $(A,B,i,j) \mapsto (B^*,-A^*,j^*,-i^*)$, so we always assume $\theta>0$ unless otherwise explicitly stated.
	We also consider the varieties $\overline{\mathfrak{M}}(n,r),\,\overline{\mathfrak{M}}^{\theta}(n,r)$ that are obtained similarly but with the space $R$ replaced by
	$\overline{R}=\mathfrak{sl}(V) \oplus \on{Hom}(V,W)$.
	We have natural isomorphisms $\mathfrak{M}^\theta(n,r)\simeq \BC^2 \times \overline{\mathfrak{M}}^\theta(n,r),\,
	\mathfrak{M}(n,r)\simeq \BC^2 \times \overline{\mathfrak{M}}(n,r)$ and projective morphisms  $\rho\colon \mathfrak{M}^{\theta}(n,r) \ra \mathfrak{M}(n,r),\,\overline{\rho}\colon \overline{\mathfrak{M}}^{\theta}(n,r) \ra \overline{\mathfrak{M}}(n,r)$. The latter morphisms are resolutions of singularities. 

We have an action of $\BC^\times \times \on{GL}(W)$ on $R$ given by  $(z,g) \cdot (A,i)=(zA,gi)$. This action naturally lifts to an action on $T^*R$ and descends to $\mathfrak{M}^\theta(n,r)$ and $\mathfrak{M}(n,r)$.
Let $T_0 \subset \on{GL}(W)$ denote a maximal torus in $\on{GL}(W)$. We set $T:=\BC^\times \times T_0$.

\ssec{Not}{Quantizations of $\mathfrak{M}(n,r)$}
We have a dilation action of $\BC^\times$ on $T^*R$ given by ${t\cdot x=t^{-1}x}$.
It descends to both $\mathfrak{M}^{\theta}(n,r)$ and $\mathfrak{M}(n,r)$. The resulting grading on $\BC[\mathfrak{M}(n,r)]$ is positive meaning that $\BC[\mathfrak{M}(n,r)]=\bigoplus_{i\geqslant 0} \BC[\mathfrak{M}(n,r)]_i$, where $\BC[\mathfrak{M}(n,r)]_i$ is the $i$-th graded component.
	The Poisson bracket on $\BC[\mathfrak{M}(n,r)]$
	has degree $-2$ with respect to this grading.
	By a quantization of $\mathfrak{M}(n,r)$ we mean an associative unital algebra $\CA$ together with an increasing filtration $\CA_i \subset \CA,\, i\in \BZ$ such that $[\CA_i,\CA_j] \subset \CA_{i+j-2}$ for any $i,\,j \in \BZ$ and an isomorphism of graded Poisson algebras $\on{gr} \CA \simeq \BC[\mathfrak{M}(n,r)]$.

	Take $c \in \BC$ and set 
	\begin{equation*}
	    \CA_c(n,r):=D(R)/\!\!/\!\!/_c\on{GL}(V):=(D(R)/[D(R)\{\xi_R-c\on{tr}\xi\,|\, \xi \in \mathfrak{gl}(V)\}])^{\on{GL}(V)},
	\end{equation*}
	where $D(R)$ is the ring of global differential operators on $R$. The algebra $\CA_c(n,r)$ has a filtration that is induced from the Bernstein filtration on $D(R)$, that is, the filtration, where both vector fields and functions on $R$ have degree $1$. There is a natural isomorphism $\BC[\mathfrak{M}(n,r)] \iso \on{gr}\CA_c(n,r)$. We analogously define quantizations $\overline{\CA}_c(n,r)$ of $\overline{\mathfrak{M}}(n,r)$. Note that $\CA_c(n,r)=D(\BC)\otimes \overline{\CA}_c(n,r)$. 
	
\subsection{Main results}	The following theorem is proved in~\cite[Theorem~1.2]{L}.
	\th{Not}\label{rep_fin_gies_param} The algebra $\overline{\CA}_c(n,r)$ has a finite dimensional representation if and only if $c=\frac{m}{n}$ with $\gcd(m,n)=1$ and $c$ is not in the interval $(-r,0)$. If that is the case then there exists a unique simple finite dimensional $\overline{\CA}_c(n,r)$-module to be denoted by $\overline{L}_{\frac{m}{n},r}$.
	\eth

	We remark that proving the existence of a finite-dimensional representation is the most difficult part of the proof of Theorem \ref{rep_fin_gies_param} in \cite{L}. We give an explicit construction of the representation $\overline{L}_{\frac{m}{n}, r}$ which, in particular, simplifies the proof of Theorem \ref{rep_fin_gies_param}. Moreover, our construction allows us to compute not only the dimension of $\overline{L}_{\frac{m}{n}, r}$, but its $\BC^{\times} \times \on{GL}_{r}$-character as well. Let us elaborate on this. We have a natural action of the group $\BC^{\times} \times \on{GL}(W) = \BC^{\times} \times \on{GL}_{r}$ on the vector space $\overline{R}$. This action commutes with the action of $\on{GL}(V)$ and thus induces an action of the group $\BC^{\times} \times \on{PGL}_{r}$ on the Gieseker variety $\overline{\mathfrak{M}}(n, r)$ as well as its resolution $\overline{\mathfrak{M}}^{\theta}(n, r)$. The action ${\BC^{\times} \times \on{PGL}_{r} \curvearrowright \overline{\mathfrak{M}}(n, r)}$ is Hamiltonian, it admits a quantum comoment map ${\Upsilon\colon \BC \oplus \mathfrak{sl}_{r} = \on{Lie}(\BC^{\times} \times \on{PGL}_{r}) \to \overline{\CA}_{c}(n, r)}$ for any value of $c$. It is easy to see that the adjoint action of $\Upsilon(\BC \oplus \mathfrak{sl}_{r})$ on $\overline{\CA}_{c}(n, r)$ is locally finite, and therefore it integrates to an action of $\BC^{\times} \times \on{PGL}_{r}$ on $\overline{\CA}_{c}(n, r)$ by algebra automorphisms. Moreover, via the map $\Upsilon$, every $\overline{\CA}_{c}(n, r)$-module becomes a $q$-graded $\mathfrak{sl}_{r}$-representation. In particular, $\overline{L}_{\frac{m}{n}, r}$ becomes a $\BC^{\times} \times \on{SL}_{r}$-representation. In the next section, we will show that the action of $\on{SL}_{r}$ extends naturally to an action of $\on{GL}_{r}$ and we will compute the $\BC^{\times} \times \on{GL}_{r}$-character of $\overline{L}_{\frac{m}{n},r}$. 
	
	\th{Not}\label{char_fd_th}    Assume $m > 0$ and $\gcd(m,n) = 1$. 
	\begin{enumerate}
	\item 
	The $\BC^{\times} \times \on{GL}_{r}$-character of $\overline{L}_{\frac{m}{n}, r}$ is 
	$$
\on{ch}_{\BC^\times \times \on{GL}_r}(\overline{L}_{\frac{m}{n},r})=\frac{1}{[n]_{q}}\sum_{\substack{\lambda \vdash m \\ r(\lambda) \leqslant \min(r;n)}}s_{\lambda}(q^{\frac{1-n}{2}}, \dots, q^{\frac{n-1}{2}})[W_{r}(\lambda)^*],
$$
where $\la$ denotes a Young diagram with $m$ boxes, $r(\la)$ is the number of rows of $\la$, $s_{\la}$ is the Schur function corresponding to $\la$, $W_r(\la)$ is the irreducible $\on{GL}_r$-module corresponding to $\la$, the square brackets denote the class in $K_0(\operatorname{Rep}(\operatorname{GL}_r))$ and $[n]_{q} := (q^{\frac{n}{2}} - q^{-\frac{n}{2}})/(q^{\frac{1}{2}} - q^{-\frac{1}{2}})$. 
	\item The 
	dimension of $\overline{L}_{\frac{m}{n},r}$ equals 
	$\frac{1}{n}{{nr+m-1}\choose{m}}$.
	\end{enumerate}
	\eth

\begin{Ex}\label{ex_sl2}
Consider the case $n=1,\, r=2,\, \theta > 0$. Then we have 
\begin{equation*}
\overline{\mathfrak{M}}(1,2)\iso \CN,\, (i,j) \mapsto (ij),\quad \overline{\mathfrak{M}}^\theta(1,2) \iso T^*(\BP^1),\, (i,j) \mapsto (ij,\on{im}i),
\end{equation*} 
where $\CN \subset \mathfrak{sl}_2$ is the nilpotent cone. Note that the filtered quantizations of $\CN$ are $\CU(\mathfrak{sl}_2)_p:=\CU(\mathfrak{sl}_2)/(C-p(p+2)),\,\, p \in \BC$, where $C:=2ef+2fe+h^2$ is the Casimir element and $\CU(\mathfrak{sl}_2)$ is the universal enveloping algebra of $\mathfrak{sl}_2$. In our notations we have $\overline{\CA}_c(1,2)=\CU(\mathfrak{sl}_2)_{c}$. To see this let us recall that 
\begin{equation*}
\overline{\CA}_c(1,2)=(D(\BC^2)/[D(\BC^2)\{x\frac{\partial}{\partial x}+y\frac{\partial}{\partial y} - c\}])^{\on{GL}_1},
\end{equation*} where $x,\,y \in {\BC^2}^*$ are standard coordinate functions and the action of $\on{GL}_1=\BC^\times$ is given by $t\cdot x=tx,\, t \cdot y=ty$. 
Then the isomorphism 
$\CU(\mathfrak{sl}_2)_c \iso \overline{\CA}_c(1,2)$ is induced by
$e \mapsto -y\frac{\partial}{\partial x},\, f \mapsto -x\frac{\partial}{\partial y},\, h \mapsto y\frac{\partial}{\partial y}-x\frac{\partial}{\partial x}$. This is nothing else but the infinitesimal action of $\mathfrak{sl}_2$ on $\BC^2$ corresponding to the standard action $\on{SL}_2 \curvearrowright \BC^2$. 
The module $\overline{L}_{m,2}$ is exactly $S^m(\BC^2)$ with trivial action of $\BC^\times$, the action of $\on{GL}_{2}$ is induced from the dual of the tautological action $\on{GL}_{2} \curvearrowright \BC^2$.

More generally, for $n=1, \,\theta > 0$ one can identify $\overline{\CA}_c(1,r)$ with a certain quotient of $\CU(\mathfrak{sl}_r)$ and the $\BC^\times \times \on{GL}_r$-module $\overline{L}_{m,r}$ is nothing else but $S^m(\BC^r)$ with trivial action of $\BC^{\times}$ and the action of $\on{GL}_{r}$ induced from the dual of the tautological action on $\BC^r$. 
\end{Ex}

	One can generalize the theorem above to the case of irreducible representations with {\it minimal support}. Let us explain what we mean by this. 
	
	Fix a  one parameter subgroup $\nu\colon 
	\BC^\times\rightarrow T$, it takes the form $(t^k, \nu_0(t))$, where  $\nu_0\colon\BC^\times \rightarrow T_0$. Assume this one parameter subgroup is generic meaning that its fixed point locus in $\mathfrak{M}^\theta(n,r)$ coincides with that for $T$. We will assume that $k>0$. To this subgroup one can assign the category $\CO_\nu(\CA_c(n,r))$ of certain $\CA_c(n,r)$-modules. If ${c\not\in (-r,0)}$ or has denominator $>n$ (or is irrational),  the irreducible objects in this category are labelled by the $r$-multipartitions of $n$. 
	We can analogously define the category $\CO_\nu(\overline{\CA}_c(n,r))$. 
	Recall that we have an isomorphism $\CA_{c}(n,r) \simeq D(\BC) \otimes \overline{\CA}_c(n,r)$. 
	It is clear that we have a label-preserving equivalence $$\CO_\nu(\overline{\CA}_c(n,r)) \iso \CO_\nu(\CA_c(n,r)), \quad  M \mapsto \BC[x] \otimes M,$$ 
	
	\noindent where $\BC[x]$ is the standard polynomial representation of the Weyl algebra $D(\BC)$, so the computation of the character of a module from $\CO_\nu(\CA_c(n,r))$ boils down to the same computation for the corresponding module in $\CO_\nu(\overline{\CA}_c(n,r))$.
	See Section~\ref{cat_O_supp} for references on categories $\CO$. 
	
	In the special case when $\nu_0(t)=\operatorname{diag}(t^{d_1},\ldots,t^{d_r})$ with $d_1\gg d_2\gg\ldots \gg d_r$, the third named author computed the GK dimensions of the irreducible modules in $\mathcal{O}_\nu(\CA_c(n,r))$, see \cite[Section 6]{L}.
	Assume that $c=\frac{m}{n}$ but $m$ and $n$ are no longer coprime. Let $d := \gcd(m,n)$. The minimal possible GK dimension of a module in $\CO_\nu(\CA_c(n,r))$ is then $d$ and the simple modules with this GK dimension are precisely those labelled by $r$-multipartitions of the form $(\varnothing,\ldots,\varnothing,n_0\lambda)$. Here $n_0:=n/d$ and $\lambda$ is a partition of $d$, by $n_0\lambda$ we mean the partition of $n$ obtained from $\lambda$ multiplying all parts by $n_0$.
	\footnote{
	Note that in~\cite{L} minimally supported modules are labeled by $(n_0\lambda,\varnothing,\ldots,\varnothing)$. The discrepancy here appears because there is a sign mistake in the proof of~\cite[Proposition~3.5]{L} that leads to a reversal of the labeling. We fix the proof of~\cite[Proposition~3.5]{L} in Proposition~\ref{fixed_points}.
	}

	Using quantum Hamiltonian reduction, in Section~\ref{repres_min_supp} we define a certain
	simple module over the algebra $\overline{\CA}_c(n,r)$ to be denoted $\overline{L}_{\frac{m}{n},r}(n_0\la)$.
	We will see that it actually lies in the category $\mathcal{O}$ corresponding to any $\nu$ of the form $(t^k, \nu_0(t))$ for $k>0$ and prove  that $\overline{L}_{\frac{m}{n},r}(n_0\la)$ is labeled by $(\varnothing,\ldots,\varnothing,n_0\la)$. 
	Recall that $\overline{L}_{\frac{m}{n}, r}(n_0\la)$ is naturally a $q$-graded $\mathfrak{sl}_{r}$-module. We will show that the action of $\mathfrak{sl}_{r}$ on $\overline{L}_{\frac{m}{n}, r}(n_0\la)$ is integrable and induces a natural action $\on{GL}_r \curvearrowright \overline{L}_{\frac{m}{n},r}(n_0\la)$. However, the $q$-grading on $\overline{L}_{\frac{m}{n},r}(n_0\la)$ does not necessarily integrate to a $\BC^\times$-action.
	
	\th{}\label{char_min_supp_th}
The $q$-graded $\on{GL}_r$-character of 
$\overline{L}_{\frac{m}{n},r}(n_0\la)$ is given by the following formula:
\begin{multline*}
\on{ch}_{q,\on{GL}_r}(\overline{L}_{\frac{m}{n},r}(n_0\la))=\\=
(1-q^{-1})\sum\limits_{\substack{\on{r}(\mu)\leqslant \on{min}(n,r)\\\mu,\beta\vdash m}}c^{\beta}_{\la,m_0}q^{-\frac{m-1}{2}+\frac{n}{m}\kappa(\beta)}\langle s_\beta\left[\frac{X}{1-q^{-1}}\right],s_{\mu}\rangle [W_r(\mu)^*],
\end{multline*}
where $\langle\,,\,\rangle$ is the Hall inner product on the algebra of symmetric functions $\La$, we use plethystic notation, $\kappa(\beta)$ is the sum of contents of all boxes of the diagram $\beta$ and the
constants $c^\beta_{\la,m_0}$ are defined as follows: 
$
		 s_{\la}(x_1^{m_0},x_2^{m_0},\ldots)=\sum_\beta c^\beta_{\la,m_0}s_\beta(x_1,x_2,\ldots),
$
where $m_{0} := m/\gcd(m,n)$.
		 \eth
	
	We remark that we have an isomorphism $\overline{\CA}_{c}(n, r) \simeq \overline{\CA}_{-c-r}(n, r)$, see for example \cite[Lemma 3.1]{L}. Thus, we will always assume $c \geqslant 0$ unless otherwise explicitly stated. 

	\subsection{Organization of the paper}
	In Section~\ref{fd_char}, we define an associative algebra $H_c(n,r)$, which is isomorphic to the rational Cherednik algebra of $\mathfrak{sl}_{n}$ for $r=1$. We study the representation theory of the algebra $H_c(n,r)$ and use it to prove Theorem~\ref{char_fd_th}.
	We also give a combinatorial interpretation to the dimension of $\overline{L}_{\frac{m}{n},r}$ in terms of parking functions, see Theorem~\ref{comb_dim}.
	In Section~\ref{loc_simple_prf}, we use Theorem~\ref{char_fd_th} to simplify the proof of the localization theorem for $\mathfrak{M}^\theta(n,r)$ given in~\cite{L}. Section~\ref{repres_min_supp} is devoted to the study of representations of $\CA_c(n,r)$ with minimal support. We construct these representations explicitly as Hamiltonian reductions of certain $D$-modules on $R$. 
	In Section~\ref{char_min_supp}, we prove Theorem~\ref{char_min_supp_th}.
	
	\subsection{Acknowledgements}
	We would like to thank Eugene Gorsky and Monica Vazirani for useful discussions. We are also grateful to the anonymous referee for helpful comments that allowed us to improve the exposition. The work of P.E. was partially supported by the NSF under grant DMS-1502244. The work of V. K. was partially supported by the Foundation for the Advancement of Theoretical Physics and Mathematics ``BASIS".

	\section{Characters of finite-dimensional representations}\label{fd_char}
	In this section, we compute the character of the representation $\overline{L}_{\frac{m}{n}, r}$ by means of a construction similar to that introduced in \cite[Section 9]{e}. This will lead us to study an algebra $H_{c}(n, r)$, explicitly defined by generators and relations, that is very similar to the rational Cherednik algebra that we obtain when setting $r = 1$. Our character computation will follow from our study of the representation theory of this algebra. For this reason, we first review the case of rational Cherednik algebras, which is well-known in the literature.

\subsection{The case $r = 1$}\label{sect:rca} In the $r = 1$ case, the algebra $\overline{\mathcal{A}}_{c}(n, 1)$ is known to be a type $A$ spherical rational Cherednik algebra, cf. \cite{gg, losev_iso}. Let us define the \emph{full} rational Cherednik algebra. The type $A$ rational Cherednik algebra (of $\mathfrak{sl}_{n}$ type) is the algebra $H_{c}(n)$ that is the quotient of the semidirect product algebra ${\BC\langle x_{1}, \dots, x_{n},  y_{1}, \dots, y_{n}\rangle \rtimes S_{n}}$ by the relations
\begin{equation}\label{eqn:rel RCA}
\begin{array}{c}
\sum_{i = 1}^{n} x_{i} = 0, \quad \sum_{i = 1}^{n}y_{i} = 0, \quad 
[x_{i}, x_{j}] = 0, \quad [y_{i}, y_{j}] = 0, \quad
 [x_{i}, y_{j}] = \frac{1}{n} - cs_{ij},
\end{array}
\end{equation}

\noindent where $s_{ij} \in S_{n}$ is the transposition $i \leftrightarrow j$ and, in the last equation, $i \neq j$. Let us remark that $H_{c}$ contains a remarkable Euler element ${\bf{h}} := \frac{1}{2} \sum (x_{i}y_{i} + y_{i}x_{i})$. This element satisfies $[{\bf{h}}, x_{i}] = x_{i}$, $[{\bf{h}}, y_{i}] = -y_{i}$ and $[{\bf{h}}, w] = 0$ for $w \in S_{n}$. In particular, it gives a
grading on $H_{c}$, and every finite-dimensional representation of $H_{c}$ is graded by eigenvalues of ${\bf{h}}$. 

We define a category $\CO_c=\CO(H_c)$ over $H_c$ as the category of finitely generated modules over $H_c$ on which elements $x_i$ act locally nilpotently. Equivalently $\CO_c$ is the category of finitely generated modules $M$ over $H_c$ such that ${\bf{h}}$ acts on $M$ with finite dimensional generalized eigenspaces and real parts of the eigenvalues of $\mathbf{h}$ on $M$ are bounded from {\em{above}}. 

\rem{}\label{rmk:hw vs lw}
Note that this definition is not the standard one (as for example in~\cite{beg, ggor}), where we ask $y_i$ to act locally nilpotently or equivalently real parts of the eigenvalues of ${\bf{h}}$ on $M$ to be bounded from {\em{below}}.
\erem

If $\tau$ is a finite dimensional module over $S_n$, then we can extend it to a module $\widetilde{\tau}$ over $\BC[x_1,\ldots,x_n] \rtimes S_n$ by letting $x_i$ act via $0$ and define the {\em{standard}} $H_c$-module $M_c(\tau)$ as follows: $M_c(\tau):= H_c \otimes_{\BC[x_1,\ldots,x_n] \rtimes S_n} \widetilde{\tau}$. One can easily check that $M_c(\tau) \in \CO_c$. 

The finite-dimensional representations of the algebra $H_{c}(n)$ have been extensively studied from algebraic, combinatorial and geometric points of view, see \cite{beg, GORS, gordon, vasserot}, for example.

\th [\cite{beg, GORS}\label{thm:fd_cherednik}
The algebra $H_{c}(n)$ admits a finite-dimensional representation if and only if $c = \frac{m}{n}$ with $\gcd(m,n) = 1$. If this is the case, there is a unique irreducible finite-dimensional representation, that we will denote $\overline{F}_{\frac{m}{n}}$, and any other finite-dimensional representation is a direct sum of copies of $\overline{F}_{\frac{m}{n}}$. Moreover, if $m > 0$, then the graded decomposition of $\overline{F}_{\frac{m}{n}}$ as an $S_{n}$-module is given by

\begin{equation}\label{eqn:character}
[\overline{F}_{\frac{m}{n}}] = \frac{1}{[m]_{q}}\bigoplus_{\lambda \vdash n} s_{\lambda}(q^{\frac{1 - m}{2}}, q^{\frac{3-m}{2}}, \dots, q^{\frac{m-1}{2}})[V_{\lambda}],
\end{equation}

\noindent where $V_{\lambda}$ is the irreducible $S_{n}$-module labeled by the partition $\lambda$, $s_{\lambda}$ is the corresponding Schur function and our normalization of quantum numbers is $[z]_{q} = \frac{q^{\frac{z}{2}} - q^{-\frac{z}{2}}}{q^{\frac{1}{2}} - q^{-\frac{1}{2}}}$.
\eth

Let us remark that the $q$-number $s_{\lambda}(q^{\frac{1-m}{2}}, \dots, q^{\frac{m-1}{2}})$ can be explicitly computed via the following hook-length formula, see e.g. \cite{reshetikhin}:
$$
s_{\lambda}(q^{\frac{1-m}{2}}, \dots, q^{\frac{m-1}{2}}) = \prod_{(i, j) \in \lambda}\frac{[m + i-j]_{q}}{[h(i,j)]_{q}},
$$
\noindent where $h(i,j)$ is the hook-length of the box $(i,j) \in \lambda$. 
In particular, $s_{\lambda}(q^{\frac{m-1}{2}}, \dots, q^{\frac{1-m}{2}}) = 0$ if the partition $\lambda$ has more than $m$ rows. 

Let us now elaborate on the connection between $H_{c}(n)$ and the algebra $\overline{\mathcal{A}}_{c}(n, 1)$. Note that the algebra $H_{c}(n)$ contains the (trivial) idempotent $\mathbf{e} := \frac{1}{n!}\sum_{w \in S_{n}}w$ of $S_{n}$. So we can form the spherical subalgebra $H_{c}^{\textrm{sph}}(n) := \mathbf{e}H_{c}(n)\mathbf{e}$. According to \cite{gg, losev_iso}, the algebras $H_{c}^{\textrm{sph}}(n)$ and $\overline{\mathcal{A}}_{c}(n,1)$ are isomorphic. Thus, we have 
$$
\overline{L}_{\frac{m}{n}, 1} = \overline{F}_{\frac{m}{n}}^{S_{n}}
$$

\noindent and the $q$-character of $\overline{L}_{\frac{m}{n},1}$ is given by
$$
\frac{1}{[m]_{q}}s_{(n)}(q^{\frac{1-m}{2}}, \dots, q^{\frac{m-1}{2}}) = \frac{1}{[m]_{q}}\left[\!\begin{matrix} n + m - 1 \\ n \end{matrix}\!\right]_{q}.
$$

\subsection{The Calaque-Enriquez-Etingof construction}\label{sect:CEE} Our approach to the computation of the character of the module $\overline{L}_{\frac{m}{n}, r}$ is based on a construction from \cite[Section~9]{e} that gives a construction of certain representations  of type $A$ rational Cherednik algebras via equivariant $D$-modules. Let us denote by $\chi\colon \mathfrak{gl}_{n} \to \BC$ the character $\chi := \frac{m}{n}\on{tr}$. Let $M$ be a $\chi$-twisted equivariant $D$-module on $\mathfrak{sl}_{n}$. 
    Recall that this means that $M$ is a $D$-module on $\mathfrak{sl}_{n}$ with a compatible $\on{GL}_{n}$-action, satisfying $\xi_{R} - \xi_{M} = \chi$ for every $\xi \in \mathfrak{gl}_{n}$.  Note that, since constant matrices act trivially on $\mathfrak{sl}_{n}$, this implies that if $\xi = \on{diag}(a, \dots, a)$, then $\xi_{M} = -am\on{Id}_M$.

It follows that the invariant space $(M \otimes \BC[\Hom(\BC^n, \BC^r)])^{\on{GL}_n}$ coincides with ${(M \otimes S^{m}\Hom(\BC^r,\BC^n))^{\on{GL}_{n}}}$. Moreover, since $\xi_{R} = 0$ for every constant matrix $\xi$ and $M$ is $\chi$-equivariant, $\xi(M \otimes S^m\Hom(\BC^r, \BC^n)) = 0$ for every constant matrix $\xi$ and it follows that 
$$
(M \otimes S^{m}\Hom(\BC^r,\BC^n))^{\on{GL}_n} = (M \otimes S^{m}\Hom(\BC^r, \BC^n))^{\mathfrak{gl}_{n}} = (M \otimes S^{m}\Hom(\BC^r, \BC^n))^{\mathfrak{sl}_{n}}.
$$

\noindent In other words, $(M \otimes S^{m}\Hom(\BC^r, \BC^n))^{\mathfrak{sl}_{n}}$ is a $\overline{\mathcal{A}}_{\frac{m}{n}}(n,r)$-module. Note that here we do not need to assume that $n$ and $m$ are coprime.

It will be convenient to study the larger space $(M \otimes \Hom(\BC^r, \BC^n)^{\otimes m})^{\mathfrak{sl}_{n}}$. To ease the notation, let us set $U := \Hom(\BC^r, \BC^n)$. We also set $F_{n, m, r}(M) := (M \otimes U^{\otimes m})^{\mathfrak{sl}_{n}}$. A priori, $F_{n, m, r}$ is a functor from the category of $\chi$-equivariant $D$-modules on $\mathfrak{sl}_{n}$ to the category of vector spaces, but we will put some extra structure on $F_{n, m, r}(M)$. First, for a matrix $P \in \mathfrak{gl}_{n}$, left multiplication by $P$ defines a map that we will denote $P\colon U \to U$. Moreover, for $i = 1, \dots, m$, we denote $(P)_{i}\colon U^{\otimes m} \to U^{\otimes m}$ the map given by left multiplication by $P$ on the $i$-th tensor factor.

We will also consider a pair of bases $(\rho_{j}),\, (\rho^{j})$ of $\mathfrak{sl}_{n}$ that are dual with respect to the trace form. We will need to make a distinction and consider $\rho_{j} \in \mathfrak{sl}_{n},\, \rho^{j} \in \mathfrak{sl}_n^{*}$. In particular, $\rho^{j}$ is a coordinate function on the space $\mathfrak{sl}_{n}$, while $\rho_{j}$ can be thought of as a differentiation with respect to $\rho^{j}$. We will think of $\rho_{j} \in D(\mathfrak{sl}_{n})$ as a degree $1$ differential operator, and of $\rho^{j} \in D(\mathfrak{sl}_{n})$ as a degree $0$ differential operator. Clearly, $[\rho_{i},\, \rho^{j}] = \delta_{ij}$. 

Finally, for $\ell_{1} \neq \ell_{2}$, let $\mathfrak{c}^{\ell_{1}, \ell_{2}}\colon U^{\otimes m} \to U^{\otimes m}$ denote  the operator that acts as $\sum_{i, j = 1}^{r} E_{ij} \otimes E_{ji}$ on the $\ell_{1}, \ell_{2}$-tensor factors of $U^{\otimes m}$. Here, $E_{ij} \in \End(\BC^r)$ is given by $(E_{ij})_{ab} = \delta_{ia}\delta_{jb}$, and $Q \in \End(\BC^r)$ acts on $U$ by multiplication by $Q^{t}$ on the right. 

\prop{} \label{prop:operators}
For $\ell = 1, \dots, m$, define the following operators on $F_{n, m, r}(M)$:
\begin{align*}
X_{\ell} := \sum_{j} \rho^{j} \otimes (\rho_{j})_{\ell}, \qquad Y_{\ell} := \frac{n}{m}\sum_{j} \rho_{j} \otimes (\rho^{j})_{\ell}.
\end{align*}

These operators satisfy the following relations:
\begin{align}
\label{sum0} \sum_{\ell = 1}^{m} X_{\ell} = 0, \quad \sum_{\ell = 1}^{m} Y_{\ell} = 0, \\
\label{comm0} [X_{\ell_{1}}, X_{\ell_{2}}] = 0, \quad [Y_{\ell_{1}}, Y_{\ell_{2}}] = 0, \\
\label{comm} [X_{\ell_{1}}, Y_{\ell_{2}}] = \frac{1}{m} - \frac{n}{m} \mathfrak{c}^{\ell_{1}, \ell_{2}}s_{\ell_{1}, \ell_{2}}, \quad \ell_{1} \neq \ell_{2}, 
\end{align}

\noindent where $s_{\ell_{1}, \ell_{2}}$ is the operator that permutes the $\ell_{1}, \ell_{2}$-tensor factors in $U^{\otimes m}$. 
\eprop
\begin{proof}
A direct computation. Relations~\eqref{sum0} follow from $\mathfrak{sl}_{n}$-invariance. Relations~\eqref{comm0} are obvious. Finally, for~\eqref{comm}, we have
$$
\begin{array}{rl}
\frac{m}{n}[X_{\ell_{1}}, Y_{\ell_{2}}] & =  \sum_{i, j} [\rho^{j} \otimes (\rho_{j})_{\ell_{1}}, \rho_{i} \otimes (\rho^{i})_{\ell_{2}}] \\
 & = \sum_{i, j} [\rho^{j}, \rho_{i}] \otimes (\rho_{j})_{\ell_{1}}(\rho^{i})_{\ell_{2}} \\
 & = \sum_{j}[\rho^{j}, \rho_{j}] \otimes (\rho_{j})_{\ell_{1}}(\rho^{j})_{\ell_{2}} \\
 & = -\sum_{j} 1 \otimes (\rho_{j})_{\ell_{1}}(\rho^{j})_{\ell_{2}}
\end{array}
$$

\noindent and the result follows from the fact that $\sum_{j} (\rho_{j})_{\ell_{1}} (\rho^{j})_{\ell_{2}} = \mathfrak{c}^{\ell_{1}, \ell_{2}}s_{\ell_{1}, \ell_{2}} - \frac{1}{n}\colon U^{\otimes m} \to U^{\otimes m}$, which is straightforward. 
\end{proof}

Note that on $(M \otimes U^{\otimes m})^{\mathfrak{sl}_{n}}$ we also have an action of $\End(\BC^r)^{\otimes m}$ by multiplying on the right by the transpose on the $U^{\otimes m}$ tensor factor, as well as an action of $S_{m}$ by permuting the tensor factors on $U^{\otimes m}$. The action of $\End(\BC^r)^{\otimes m}$ commutes with the action of $X_{1}, \dots, X_{m},\, Y_{1}, \dots, Y_{m}$, and $S_{m}$ satisfies the obvious commutation relations with $X$'s, $Y$'s and $\End(\BC^r)^{\otimes m}$. This motivates the following definition.

\begin{Def}
Let $m, r \in \BZ_{> 0}$ and $c \in \BC$. We define the algebra $H_{c}(m, r)$ to be the quotient of the semi-direct product algebra $$(\BC\langle x_{1}, \dots, x_{m}, y_{1}, \dots, y_{m}\rangle \otimes \End(\BC^{r})^{\otimes m}) \rtimes S_{m}$$ by the relations
\begin{align}
\label{sum} \sum x_{\ell} = 0, \quad \sum y_{\ell} = 0, \\
\label{comm rca0}[x_{\ell_{1}}, x_{\ell_{2}}] = 0, \quad [y_{\ell_{1}}, y_{\ell_{2}}] = 0, \\
\label{comm rca}[x_{\ell_{1}}, y_{\ell_{2}}] = \frac{1}{m} - c \sum_{i,j = 1}^{r} (E_{ij})_{\ell_{1}}(E_{ji})_{\ell_{2}} s_{\ell_{1},\ell_{2}}, \quad \ell_{1} \neq \ell_{2}.
\end{align}
\end{Def}

\begin{Ex}
When $r = 1$, $H_{c}(m, 1)$ is nothing but the rational Cherednik algebra $H_{c}(m)$ defined in Section~\ref{sect:rca}. 
\end{Ex}

Then Proposition~\ref{prop:operators} can be reinterpreted as follows.

\prop{}\label{prop:functor}
The correspondence sending $M$ to $F_{n, m, r}(M)$ defines a functor 
\begin{equation*}
F_{n, m, r}\colon D(\mathfrak{sl}_{n})\on{-mod}^{\on{GL}_{n}, \chi} \to H_{\frac{n}{m}}(m, r)\on{-mod}.
\end{equation*}
\eprop

\subsection{The Dunkl embedding for $H_{c}(m, r)$}\label{sect:dunkl} Proposition~\ref{prop:functor} motivates the study of the algebra $H_{c}(m, r)$ and its representation theory. As it turns out,  the algebras $H_{c}(m, r)$ and $H_{c}(m, 1)$ are Morita equivalent. Before stating our result, let us establish some structural properties of the algebra $H_{c}(m, r)$. In particular, we will define a polynomial representation for $H_{c}(m, r)$. 

Let  $\mathfrak{h}$ denote the $(m-1)$-dimensional reflection representation of the symmetric group $S_{m}$. Recall that the algebra $H_{c}(m, 1)$ acts on $\BC[\mathfrak{h}]$, with $x_{1}, \dots, x_{m}$ acting by multiplication and $y_{1}, \dots, y_{m}$ acting by Dunkl operators. Let us denote these operators on $\BC[\mathfrak{h}]$ by ${\sf{x}}_{1}, \dots, {\sf{x}}_{m}, {\sf{y}}_{1}, \dots, {\sf{y}}_{m}$. 

\prop{}\label{prop:dunkl}
The algebra $H_{c}(m, r)$ acts on the space $\BC[\mathfrak{h}] \otimes (\BC^{r})^{\otimes m}$ as follows:
\begin{itemize}
\item $x_{1}, \dots, x_{m}, y_{1}, \dots, y_{m}$ act by ${\sf{x}}_1 \otimes 1, \dots, {\sf{x}}_m \otimes 1, {\sf{y}}_1 \otimes 1, \dots, {\sf{y}}_m \otimes 1$, respectively.
\item $A_{1} \otimes \cdots \otimes  A_{m} \in \End(\BC^r)^{\otimes m}$ acts by $1 \otimes A_{1} \otimes \cdots \otimes A_{m}$.
\item $S_{m}$ acts diagonally. 
\end{itemize}
\eprop
\begin{proof}
A direct computation, the main point here is that $\sum E_{ij} \otimes E_{ji}$ acts on $\BC^{r} \otimes \BC^{r}$ by switching the tensor factors.
\end{proof}

Let us now denote by $\mathfrak{h}^{\textrm{reg}} \subset \mathfrak{h}$ the locus where the $S_{m}$-action is free, note that this is a principal open set. It is well-known that the algebra $H_{c}(m, 1)$ admits an embedding $H_{c}(m, 1) \hookrightarrow D(\mathfrak{h}^{\textrm{reg}}) \rtimes S_{m}$, by interpreting $x_{1}, \dots, x_{m}$ as functions on $\mathfrak{h}$ and $y_{1}, \dots, y_{m}$ as Dunkl operators. Thanks to Proposition~\ref{prop:dunkl}, the algebra $H_{c}(m, r)$ admits a similar embedding to $(D(\mathfrak{h}^{\textrm{reg}}) \otimes \End(\BC^r)^{\otimes m}) \rtimes S_{m}$. Since the action of ${(D(\mathfrak{h}^{\textrm{reg}}) \otimes \End(\BC^r)^{\otimes m}) \rtimes S_{m}}$ on $\BC[\mathfrak{h}] \otimes (\BC^r)^{\otimes m}$ is faithful, this has the following consequence. 

\prop{}[PBW property for $H_{c}(m, r)$]\label{prop:PBW}
Multiplication induces an isomorphism $S(\mathfrak{h}^{*}) \otimes \left(\End(\BC^{r})^{\otimes m} \rtimes S_{m}\right) \otimes S(\mathfrak{h}) \xrightarrow{\sim} H_{c}(m, r)$. 
\eprop

\cor{}\label{cor:morita}
Let $E := E_{11}^{\otimes m} \in \End(\BC^r)^{\otimes m}$. Then $E$ is an idempotent, $H_{c}(m, r)E H_{c}(m, r) = H_{c}(m, r)$ and $EH_{c}(m, r)E \cong H_{c}(m, 1)$. In particular, $H_{c}(m, r)$ is Morita equivalent to the usual type $A$ rational Cherednik algebra $H_{c}(m)$.
\ecor
\begin{proof}
The first two assertions are clear. For the last assertion, define a map $\BC\langle x_{1}, \dots, x_{m}, y_{1}, \dots, y_{m}\rangle \rtimes S_{m} \to E H_{c}(m, r) E$ by sending $x_{\ell} \mapsto Ex_{\ell}E = x_{\ell}E$, $y_{\ell} \mapsto E y_{\ell}E = y_{\ell}E$ and $w \mapsto EwE = wE$ for $w \in S_{m}$. It is easy to check that this map factors through an algebra homomorphism $H_{c}(m, 1) \to E H_{c}(m, r)E$. That it is an isomorphism follows from the PBW property.   
\end{proof}

\subsection{Representation theory of $H_{c}(m, r)$} Thanks to Corollary~\ref{cor:morita}, the algebras $H_{c}(m, r)$ and $H_{c}(m, 1)$ are Morita equivalent. In fact, we have a very concrete realization of this Morita equivalence, that generalizes Proposition~\ref{prop:dunkl}.

\prop{}\label{prop:morita}
Let $N \in H_{c}(m, 1)\operatorname{-mod}$. Denote the action of $x_{1}, \dots, x_{m},\, y_{1}, \dots, y_{m} \in H_{c}(m, 1)$ on $N$ by ${\sf{x}}_1, \dots {\sf{x}}_m,\, {\sf{y}}_1, \dots {\sf{y}}_m$, respectively. Define $\Phi(N) := N \otimes (\BC^r)^{\otimes m}$. Then $\Phi(N)$ becomes a $H_{c}(m, r)$-module by the same formulas as those in Proposition~\ref{prop:dunkl}, and $\Phi\colon H_{c}(m, 1)\on{-mod} \to H_{c}(m, r)\on{-mod}$ is an inverse to the functor $E\colon H_{c}(m, r)\on{-mod} \to H_{c}(m, 1)\on{-mod}$. 
\eprop
\begin{proof}
That the formulas do define an action of $H_{c}(m, r)$ on $\Phi(N)$ is a straightforward direct computation. Note that $E(N \otimes (\BC^r)^{\otimes m}) = N$. So the functor $N \mapsto \Phi(N)$ is a right inverse to the Morita equivalence of Corollary~\ref{cor:morita} and the result follows. 
\end{proof}

Thanks to Theorem~\ref{thm:fd_cherednik} we can see the following. 

\cor{}\label{cor:invariants}
The algebra $H_{c}(m, r)$ admits a finite-dimensional representation if and only if $c = \frac{n}{m}$ with $\gcd(m,n) = 1$. Moreover, the unique irreducible finite-dimensional representation is $\overline{F}_{\frac{n}{m}} \otimes (\BC^{r})^{\otimes m}$ where, recall, $\overline{F}_{\frac{n}{m}}$ is the unique irreducible finite-dimensional representation of $H_{\frac{n}{m}}(m, 1)$.
\ecor

Let us now see that we can get the unique irreducible finite-dimensional representation of $H_{\frac{n}{m}}(m, r)$ from a $D$-module on $\mathfrak{sl}_{n}$ via the functor $F_{n, m, r}$. 

\prop{}\label{prop:morita intertwines}
The equivalence in Corollary~\ref{cor:morita} intertwines the functors $F_{n, m, r}$ and $F_{n, m, 1}$, that is, the following diagram commutes:
$$
\xymatrix{ & D(\mathfrak{sl}_{n})\on{-mod}^{\on{GL}_{n}, \chi} \ar[dl]_{F_{n, m, r}} \ar[dr]^{F_{n, m, 1}} & \\ H_{\frac{n}{m}}(m, r)\on{-mod} \ar[rr]^{M \mapsto EM} & & H_{\frac{n}{m}}(m, 1)\on{-mod} }
$$
\eprop
\begin{proof}
Since $\Hom(\BC^r, \BC^n)E_{11} = \BC^n$, it follows that ${EF_{n, m, r}(M) = (M \otimes (\BC^n)^{\otimes m})^{\mathfrak{sl}_{n}}}$, with the correct formulas for the action of $x_{1}, \dots, x_{m},\, y_{1}, \dots, y_{m}$. 
\end{proof}

Now assume $m$ and $n$ are coprime and let $O$ be the regular nilpotent orbit in $\mathfrak{sl}_{n}$. Consider the rank one local system on $O$ that corresponds to the representation of the center $Z(\on{SL}_{n}) \subset \on{SL}_n$ given by $\on{diag}(z, \dots, z) \to z^{-m}$, and let $M(O)$ be its minimal extension. This is an $\on{SL}_{n}$-equivariant $D$-module on $\mathfrak{sl}_{n}$. Extend the $\on{SL}_{n}$-action on $M(O)$ to a $\on{GL}_{n}$-action by requiring a matrix $\on{diag}(a, \dots, a)$ to act by multiplication by $a^{-m}$. This makes $M(O)$ a $\chi$-equivariant $D$-module on $\mathfrak{sl}_{n}$. The next result is now a consequence of Proposition~\ref{prop:morita intertwines} and \cite[Section 9.12]{e}  

\cor{}\label{cor:reg nilpotent}
The module $F_{n, m, r}(M(O))$ is the unique irreducible finite-dimensional representation of $H_{\frac{n}{m}}(m, r)$. In particular, $F_{n, m, r}(M(O)) = \overline{F}_{\frac{n}{m}} \otimes (\BC^r)^{\otimes m}$. 
\ecor

\rem{}\label{rmk:Fourier}
Let us compare the functor $F_{n, m, 1}$ to that used in the work of Calaque-Enriquez-Etingof \cite{e}.  It is easy to see that, if $M$ is a $\chi$-equivariant $D$-module on $\mathfrak{sl}_{n}$ supported on the nilpotent cone $\mathcal{N}$, the action of $x_{1}, \dots, x_{m} \in H_{\frac{n}{m}}(m,r)$ on $F_{n, m, r}(M)$ is locally nilpotent and the action of the Euler element $\mathbf{h} = \frac{1}{2}\sum x_{i}y_{i} + y_{i}x_{i}$ is locally finite. In other words, the module $F_{n, m, r}(M)$ belongs to the category $\mathcal{O}$ of \emph{highest} weight modules.  

In \cite{e}, the authors consider the functor $F^{*}_{n, m, 1} := F_{n, m, 1} \circ \mathcal{F}$, where $\mathcal{F}$ is the usual Fourier transform on $D$-modules. The reason is that, if $M$ is supported on the nilpotent cone $\mathcal{N}$, then the action of $y_{1}, \dots, y_{n}$ on $F^{*}_{n, m, 1}(M)$ is locally nilpotent, so $F^{*}_{n, m, 1}(M)$ belongs to the category of \emph{lowest}-weight modules for $H_{\frac{n}{m}}(m, 1)$, which is more common in the Cherednik algebra literature, see Remark \ref{rmk:hw vs lw} above. Note, however, that since the $D$-module $M(O)$ considered in Corollary \ref{cor:reg nilpotent} is cuspidal, both $M(O)$ and $\mathcal{F}(M(O))$ are supported on the nilpotent cone, and the same reasoning as in \cite[Section 9.12]{e} implies that $F_{n, m, 1}(M(O))$ is a finite-dimensional irreducible representation of $H_{\frac{n}{m}}(m, 1)$. 

\erem

\subsection{Spherical subalgebra} Note that $H_{c}(m,r)$ contains the idempotent $\mathbf{e} := \frac{1}{m!}\sum_{w \in S_{m}} w$, so we have the spherical subalgebra $H_{c}^{\textrm{sph}}(m,r) := \mathbf{e}H_{c}(m, r)\mathbf{e}$. As usual, we have a quotient functor $N \mapsto \mathbf{e}N = N^{S_{m}}$, $H_{c}(m, r)\operatorname{-mod} \to H_{c}^{\textrm{sph}}(m,r)\operatorname{-mod}$ that is an equivalence provided $\mathbf{e}N \neq 0$ for every $N \in H_{c}(m, r)\operatorname{-mod}$. 

\prop{}\label{prop:spherical}
Assume that $c \notin (-1, 0)$ or that $r \geqslant m$. Then the algebras $H_{c}(m, r)$ and $H_{c}^{\textrm{sph}}(m, r)$ are Morita equivalent. In particular, if $c = \frac{n}{m} > 0$ with $\gcd(m,n) = 1$, we have that $H_{c}^{\textrm{sph}}(m, r)$ admits a unique irreducible finite-dimensional representation, given by $\mathbf{e}(\overline{F}_{\frac{n}{m}} \otimes (\BC^r)^{\otimes m}) = (\overline{F}_{\frac{n}{m}} \otimes (\BC^r)^{\otimes m})^{S_{m}}$. 
\eprop
\begin{proof}
The case $c \notin (-1, 0)$ follows from~\cite[Corollary 4.2]{be}, as follows. We need to show that $\mathbf{e}\widetilde{N} = \widetilde{N}^{S_{m}}\not= 0$ for every $\widetilde{N} \in H_{c}(m, r)\on{-mod}$. By Proposition \ref{prop:morita}, $\widetilde{N} = N \otimes (\BC^r)^{\otimes m}$ for some $N \in H_{c}(m, 1)\on{-mod}$. Now, by~\cite[Corollary 4.2]{be}, $N^{S_{m}} \neq 0$ provided $c \not\in (-1, 0)$. Then $0 \neq N^{S^{m}} \otimes S^{m}(\BC^r) \subseteq \widetilde{N}^{S_{m}}$ and we are done. \\
If $r \geqslant m$, by Schur-Weyl duality we have that every irreducible representation of $S_{m}$ appears with nonzero multiplicity in $(\BC^r)^{\otimes m}$. So, using the notation of the previous paragraph, $\widetilde{N}^{S_{m}} = (N \otimes (\BC^r)^{\otimes m})^{S_{m}} \neq 0$ for every nonzero $\widetilde{N} \in H_{c}(m, r)\on{-mod}$ and the result now follows from Proposition~\ref{prop:morita}.
\end{proof}

\subsection{Functor $F^{\sph}_{n, m, r}$ vs. Hamiltonian reduction} Note that we have a functor $F^{\textrm{sph}}_{n, m, r}\colon D(\mathfrak{sl}_{n})\operatorname{-mod}^{\on{GL}_{n}, \chi} \to H_{\frac{n}{m}}^{\textrm{sph}}(m, r)\operatorname{-mod}$ which is defined by $F_{n, m, r}^{\textrm{sph}} := \mathbf{e}F_{n, m, r}$. By the definition of the functor $F_{n, m, r}$, we have that $F^{\textrm{sph}}_{n, m, r}(M) = (M \otimes S^{m}U)^{\mathfrak{sl}_{n}}$. 

On the other hand, we have a Hamiltonian reduction functor 
\begin{equation*}
\mathbb{H}\colon D(\mathfrak{sl}_{n})\operatorname{-mod}^{\on{GL}_{n}, \chi} \to \overline{\mathcal{A}}_{\frac{m}{n}}(n, r)\operatorname{-mod},
\end{equation*} 
given by taking the invariant space $\mathbb{H}(M) := (M \otimes \BC[\Hom(\BC^n, \BC^r)])^{\on{GL}_{n}}$. Thanks to the discussion at the beginning of Section~\ref{sect:CEE} we have that, as vector spaces, ${\mathbb{H}(M) = F^{\textrm{sph}}_{n, m, r}(M)}$ for every $M \in D(\mathfrak{sl}_{n})\operatorname{-mod}^{\on{GL}_{n}, \chi}$.

We claim that even more is true. Note that by Propositions~\ref{prop:morita intertwines} and~\ref{prop:spherical}, another formula for the functor $F^{\textrm{sph}}_{n, m, r}$ is given by $$
F^{\textrm{sph}}_{n, m, r}(M) = (F_{n, m, 1}(M) \otimes (\BC^{r})^{\otimes m})^{S_{m}}
$$

\noindent the space $F_{n, m, 1}(M)$ has a grading coming from the Euler operator in $H_{\frac{n}{m}}(m, 1)$, while the space $(\BC^{r})^{\otimes m}$ has a natural $\operatorname{GL}_{r}$-action. Both the action of the Euler operator and the $\operatorname{GL}_{r}$-action commute with the action of $S_{m}$, so we get commuting $q$-grading and $\operatorname{GL}_{r}$-action on $F^{\textrm{sph}}_{n, m, r}(M)$. 

\rem{}\label{rmk:q-grading}
It follows from~\cite[Section~8]{e} that the action of the Euler element $\mathbf{h}$ on $F_{n,m,1}(M)$ extends to an action of $\mathfrak{sl}_2$ on $F_{n,m,1}(M)$ so if $F_{n,m,1}(M)$ is finite-dimensional, then our $q$-grading integrates to an action of $\BC^\times$ on $F_{n,m,1}(M)$.
\erem

\th{}\label{thm:main1}
Let $M \in D(\mathfrak{sl}_{n})\operatorname{-mod}^{\on{GL}_{n}, \chi}$. Then, as $q$-graded $\mathfrak{sl}_{r}$-modules,
$$
\mathbb{H}(M) = (F_{n, m, 1}(M) \otimes (\BC^{r*})^{\otimes m})^{S_{m}}.
$$

\noindent In particular, the $\mathfrak{sl}_{r}$-action on $\mathbb{H}(M)$ integrates to a $\on{SL}_{r}$-action that can be extended to a $\on{GL}_{r}$-action in a natural way. The $q$-grading integrates to a $\BC^{\times}$-action provided the same is true for the $H_{\frac{n}{m}}(m, 1)$-module $F_{n, m, 1}(M)$. 
\eth
\begin{proof}
Let us deal with the $\mathfrak{sl}_{r}$-action. First, we will consider the action of $\End(\BC^r)^{\otimes m}$. The action of $\End(\BC^r)^{\otimes m}$ on $\mathbb{H}(M) = (M \otimes S^mU)^{\mathfrak{sl}_{n}}$ is induced from the action on $U = \Hom(\BC^r, \BC^n)$, so the embedding $\mathbb{H}(M) \hookrightarrow (M \otimes U^{\otimes m})^{\mathfrak{sl}_{n}} = F_{n, m, r}(M)$ is $\End(\BC^r)^{\otimes r}$-equivariant. Now, the isomorphism $F_{n, m, r}(M) = F_{n, m, 1}(M) \otimes (\BC^r)^{\otimes m}$ is that of $H_{\frac{n}{m}}(m, r)$-modules, and is therefore both $S_{m}$ and $\End(\BC^r)^{\otimes m}$-equivariant. Thus, as $\End(\BC^r)^{\otimes m}$-modules, we have that $\mathbb{H}(M)$ gets identified with the $S_{m}$-invariant part of $F_{n, m, 1}(M) \otimes (\BC^r)^{\otimes m}$. 

Note, however, that the action of $\End(\BC^r)^{\otimes m}$ on $U^{\otimes m}$ is by multiplication by the transpose on the right, and thus it is not compatible with the action of $\mathfrak{sl}_{r}$. To fix this, one needs to first apply the automorphism $\xi \mapsto -\xi$ of $\mathfrak{sl}_{r}$. This induces the antipodal map at the level of the enveloping algebra $\mathcal{U}(\mathfrak{sl}_{r})$ and thus we get ${\mathbb{H}(M) = (F_{n, m, 1} \otimes (\BC^{r*})^{\otimes m})^{S_{m}}}$ as $\mathfrak{sl}_{r}$-modules, as desired. The claim about the integrability of the $\mathfrak{sl}_{r}$-action follows easily. The action of $\on{SL}_{r}$ on $(\BC^{r*})^{\otimes m}$ naturally extends to an action of $\on{GL}_{r}$.

Note that the isomorphism $F_{n, m, r}(M) \cong F_{n, m, 1}(M) \otimes (\BC^{r})^{\otimes m}$ is that of $H_{\frac{n}{m}}(m, r)$-modules and therefore preserves $q$-gradings. So it remains to show that the embedding $\mathbb{H}(M) \hookrightarrow F_{n, m, 1}(M) \otimes (\BC^r)^{\otimes m}$ that we produced in the first paragraph of this proof also preserves the $q$-gradings. The $q$-grading on $\mathbb{H}(M)$ is induced by the action of the operator $\sum_{j}\rho^{j}\rho_{j}$ where, recall, $\rho_{j} \in \mathfrak{sl}_{n}$, $\rho^{j} \in \mathfrak{sl}_{n}^{*}$ are a pair of dual bases. The coincidence of the actions now follows from the formulas in Proposition \ref{prop:operators}, see e.g.~\cite[Proposition 8.7]{e}. Finally, as $q$-graded $S_{m}$-modules we clearly have $F_{n, m, 1} \otimes (\BC^r)^{\otimes m} = F_{n, m, 1} \otimes (\BC^{r*})^{\otimes m}$ and the claim follows. 
\end{proof}

Note that, since we are assuming that $c = \frac{m}{n} \geqslant 0$, we are always in one of the cases considered in Proposition \ref{prop:spherical} and therefore we have that $\mathbb{H}(M) \neq 0$ if and only if $F_{n, m, 1}(M) \neq 0$.

Now let $O$ be the regular nilpotent orbit in $\mathfrak{sl}_{n}$, and $M(O)$ the irreducible $\chi$-equivariant $D$-module on $\mathfrak{sl}_{n}$ associated to $O$. Then $M(O) \otimes \BC[\Hom(\BC^n, \BC^r)]$ is an irreducible $\chi$-equivariant $D(\overline{R})$-module and, since $\mathbb{H}$ is a quotient functor, ${\mathbb{H}(M(O)) = (F_{\frac{n}{m}} \otimes (\BC^{r})^{\otimes m})^{S_{m}} \neq 0}$ is an irreducible $\overline{\mathcal{A}}_{\frac{m}{n}}(n, r)$-module, which is finite-dimensional and therefore isomorphic to $\overline{L}_{\frac{m}{n}, r}$. By Remark \ref{rmk:q-grading} the $q$-grading on $\overline{L}_{\frac{m}{n}, r}$ integrates to a $\BC^{\times}$-action. Then we obtain.

\cor{}\label{cor:L}
As $\BC^\times\times \on{GL}_{r}$-modules,
\begin{equation*}
\overline{L}_{\frac{m}{n}, r} \cong (\overline{F}_{\frac{n}{m}} \otimes (\BC^{r*})^{\otimes m})^{S_{m}}.
\end{equation*}
\ecor

Let us explicitly compute the $\BC^\times \times\on{GL}_{r}$-character of $\overline{L}_{\frac{m}{n}, r}$. From~\eqref{eqn:character} we have that the $\BC^{\times}$-character of $\overline{F}_{\frac{n}{m}}$ is
$$
\frac{1}{[n]_{q}}\bigoplus_{\lambda \vdash m}s_{\lambda}(q^{\frac{1 - n}{2}}, \dots, q^{\frac{n-1}{2}})V_{\lambda}.
$$

\noindent Now, for a partition $\lambda$ with at most $r$ parts, denote by $W_{r}(\lambda)$ the irreducible $\on{GL}_{r}$-summand of $(\BC^r)^{\otimes m}$ indexed by $\lambda$. Then, by Schur-Weyl duality, we can express the character of $\overline{L}_{\frac{m}{n}, r}$ by
$$
\frac{1}{[n]_{q}}\bigoplus_{\substack{\lambda \vdash m \\ r(\lambda) \leqslant r}}s_{\lambda}(q^{\frac{1-n}{2}}, \dots, q^{\frac{n-1}{2}})W_{r}(\lambda)^{*},
$$

\noindent where $r(\lambda)$ is the number of rows of $\lambda$.  Even more is true. The Schur function $s_{\lambda}(q^{\frac{1-n}{2}}, \dots, q^{\frac{n-1}{2}})$ is the graded dimension of the representation $W_{n}(\lambda)$ of $\mathfrak{gl}_{n}$ and therefore vanishes when $r(\lambda) > n$. Then we obtain our character formula
\begin{equation}\label{form_char}
\overline{L}_{\frac{m}{n}, r} = \frac{1}{[n]_{q}} \bigoplus_{\substack{\lambda \vdash m \\ r(\lambda) \leqslant \min(n;r)}}s_{\lambda}(q^{\frac{1-n}{2}}, \dots, q^{\frac{n-1}{2}})W_{r}(\lambda)^{*}.
\end{equation}

Note that, if we ignore the $\frac{1}{[n]_{q}}$-factor in~\eqref{form_char} we have the graded character of the $\on{GL}_{n} \times \on{GL}_{r}$-representation $S^{m}(\BC^{n} \otimes \BC^{r})$, where the grading only affects the $\BC^{n}$-part. It follows that the 
dimension of $\overline{L}_{\frac{m}{n},r}$ is $\frac{1}{n}{nr+m-1\choose{m}}=\frac{1}{n}\on{dim}S^m(\BC^n \otimes \BC^r)$. 

\cor{}\label{cor:symmetric}
The $\BC^{\times}$-character of $\overline{L}_{\frac{m}{n}, r}$ is a Laurent polynomial in $q$ that is symmetric under the change of variables $q \leftrightarrow q^{-1}$, i.e. $\on{ch}_{q}(\overline{L}_{\frac{m}{n}, r}) = \on{ch}_{q^{-1}}(\overline{L}_{\frac{m}{n}, r})$. Moreover, its degree is $(n-1)(m-1)/2$ and its leading coefficient is $\binom{r + m - 1}{m}$.
\ecor
\begin{proof}
It is known that the action of the Euler element $\mathbf{h}$ on $\overline{F}_{\frac{n}{m}}$ extends to an action of $\mathfrak{sl}_{2}$ on $\overline{F}_{\frac{n}{m}}$ and, moreover, this action commutes with the action of $S_{m}$. Thus, the fact that $\on{ch}_{q}(\overline{L}_{\frac{m}{n}, r}) = \on{ch}_{q^{-1}}(\overline{L}_{\frac{m}{n}, r})$ follows immediately from Corollary \ref{cor:L}. From \eqref{form_char} it follows that the degree of $\on{ch}_{q}(\overline{L}_{\frac{m}{n}, r})$ is independent of $r$, and in the case $r = 1$ it is easy to see that the degree is precisely $(m-1)(n-1)/2$. Finally, from \eqref{form_char} it is also easy to see that the leading coefficient is $\dim(S^{m}(\BC^{r*})) = \binom{r + m - 1}{m}$. 
\end{proof}

\begin{Ex}
		Consider now the example $m=3, n=r=2$. We have 
		\begin{equation*}
        \overline{L}_{3/2,2}=	W^{*}_{_{\,\tiny \Yvcentermath1 \yng(2,1)}} + (q+q^{-1})W^{*}_{_{\,\tiny \Yvcentermath1 \yng(3)}},\quad \on{dim}\overline{L}_{3/2,2}=10.
		\end{equation*}
		\end{Ex}

\rem{Comp}\label{zero_dim}
		Note that the roots of $\frac{1}{n}{{nr+m-1}\choose{m}}=\frac{1}{n}{{nr+m-1}\choose{nr-1}}$ considered as a polynomial of $m$ (with fixed $n$) are precisely $-1,\ldots,-nr+1$, i.e. we have $\on{dim}\overline{L}_{\frac{m}{n},r}=0$ for $-rn < m < 0$. This exactly corresponds to the fact that the algebra ${\overline{\CA}_c(n,r)=D(\mathfrak{sl}_n \times (\BC^n)^{\oplus r})/\!\!/\!\!/_{\frac{m}{n}}\on{GL}_n}$ has infinite homological dimension and does not have nontrivial finite dimensional representations in these cases, see~\cite[Theorems~1.1,\,1.2]{L}. We will use these observations in Section~\ref{loc_simple_prf} to greatly simplify the most technical parts of the proof of Theorem 1.1 in loc.cit.
		\erem

\rem{}\label{rmk:generating function}
Let $\{E_{ii}\}$ be the standard basis of $\mathfrak{t}$, i.e. $\on{diag}(a_1,\ldots,a_r)=\sum_i a_iE_{ii}$.
We can define the $T$-character of $\overline{L}_{\frac{n}{m},r}$ as follows:
	    \begin{equation*}
	    g(q,q_1,\ldots,q_r):=\on{Tr}_{\overline{L}_{\frac{n}{m},r}}(q^{{\bf{h}}}q^{E_{11}}_1\ldots q_r^{E_{rr}}).
	    \end{equation*}
It is easy to deduce from the Cauchy identity that $g(q,q_1,\ldots,q_r)$ is the coefficient in front of $z^m$ of the following ``generating function"\footnote{Note that if $m$ and $n$ are not coprime, the coefficient in front of $z^{m}$ in $D(z)$  is not the character of a finite-dimensional representation and, in fact, does not need to be in $\BZ[q, q^{-1}, q_{j}, q_{j}^{-1}\,|\,1 \leqslant j \leqslant r]$. This is the reason why we write ``generating function'' inside quotation marks.}
\begin{equation}\label{eqn:generating function}
D(z)=\frac{1}{[n]_q} \prod_{\substack{1\leqslant i \leqslant n \\ 1 \leqslant j \leqslant r}} \frac{1}{1-zq^{\frac{n+1-2i}{2}}q_j^{-1}}.
\end{equation}
\erem

\ssec{}{Semistandard parking functions and higher rank Catalan numbers} The goal of this section is to give a combinatorial interpretation of $\dim \overline{L}_{\frac{m}{n}, r}$. 

\defe{}
We will call the number $C_{\frac{m}{n}, r} := \dim(\overline{L}_{\frac{m}{n}, r}) = \frac{1}{n}\binom{nr + m - 1}{m}$ the rank $r$ rational $\frac{m}{n}$-Catalan number.
\edefe

Let us, first, review the case $r = 1$ which is well-known. Since $m$ and $n$ are coprime, the number $C_{\frac{m}{n}, 1} = \frac{1}{n} \binom{n + m - 1}{m} = \frac{1}{n+m}\binom{n+m}{m}$ counts the number of \emph{$\frac{m}{n}$-Dyck paths}, that is, paths in $\mathbb{Z}^{2}$ from $(0,0)$ to $(n, m)$ that use only steps in the directions $(1,0)$ and $(0,1)$, and that always stay above the diagonal line $y = \frac{m}{n}x$. We will denote by $\mathcal{D}_{\frac{m}{n}}$ the set of $\frac{m}{n}$-Dyck paths. 

Now let $D \in \mathcal{D}_{\frac{m}{n}}$. A \emph{vertical run} of $D$ is a maximal collection of consecutive vertical steps. Let $a_{1}, \ldots, a_{\ell}$ be the lengths of the vertical runs of $D$. Note that $a_{1} + \ldots + a_{\ell} = m$. The following result will be very important for us.

\begin{lemma}{}\label{lemma:dyck}
As a representation of $S_{m}$,

$$
\overline{F}_{\frac{n}{m}} = \bigoplus_{D \in \mathcal{D}_{\frac{m}{n}}} \on{Ind}_{S_{a_{1}} \times \cdots \times S_{a_{\ell}}}^{S_{m}} \on{triv}.
$$
\end{lemma}
\begin{proof}
Follows from~\cite[Proposition 1.7]{beg} and~\cite[Corollary 4]{ALW}.
\end{proof}

A consequence of Lemma~\ref{lemma:dyck} is that, as $\on{GL}_{r}$-modules,
\begin{equation}\label{eq:comb character}
\begin{array}{rl}
\overline{L}_{\frac{m}{n}, r} & = (\overline{F}_{\frac{n}{m}} \otimes (\BC^{r*})^{\otimes m})^{S_{m}} \\
& = \Hom_{S_{m}}(\overline{F}_{\frac{n}{m}}, (\BC^{r*})^{\otimes m}) \\
 & = \bigoplus_{D \in \mathcal{D}_{\frac{m}{n}}} \Hom_{S_{m}}(\on{Ind}_{S_{a_{1}} \times \cdots \times S_{a_{\ell}}}^{S_{m}}\on{triv}, (\BC^{r*})^{\otimes m}) \\
& = \bigoplus_{D \in \mathcal{D}_{\frac{m}{n}}} \Hom_{S_{a_{1}} \times \cdots \times S_{a_{\ell}}}(\on{triv}, \on{Res}^{S_{m}}_{S_{a_{1}} \times \cdots \times S_{a_{\ell}}}((\BC^{r*})^{\otimes m})) \\
& = \bigoplus_{D \in \mathcal{D}_{\frac{m}{n}}} S^{a_{1}}(\BC^{r*}) \otimes \cdots \otimes S^{a_{\ell}}(\BC^{r*}).
\end{array}
\end{equation}

\rem{}
Note that \eqref{eq:comb character} gives another expression for the $\on{GL}_{r}$-character of $\overline{L}_{\frac{m}{n}, r}$. 
\erem

We will use equation \eqref{eq:comb character} to give a combinatorial interpretation of $C_{\frac{m}{n}, r}$. The key concept is the following. 

\defe{}\label{def:parking}
A rank $r$ semistandard $\frac{m}{n}$-parking function consists of a pair $(D, \varphi)$, where $D$ is an $\frac{m}{n}$-Dyck path, $D \in \mathcal{D}_{\frac{m}{n}}$ and $\varphi\colon \{\text{vertical steps of} \; D\} \to \{1, \dots, r\}$ is a function that is weakly increasing along each vertical run, reading from top-to-bottom. We will denote by $\mathcal{PF}^{r}_{\frac{m}{n}}$ the set of rank $r$ semistandard $\frac{m}{n}$-parking functions. 
\edefe

\rem{}
Recall that an $\frac{m}{n}$-parking function consists of an $\frac{m}{n}$-Dyck path together with a bijection from its set of vertical steps to $\{1, \dots, m\}$ that is strictly increasing along each vertical run. This explains the terminology in Definition~\ref{def:parking}.
\erem

\th{}\label{comb_dim}
Assume $m$ and $n$ are coprime. Then $|\mathcal{PF}^{r}_{\frac{m}{n}}| = C_{\frac{m}{n}, r}$. 
\eth
\begin{proof}
It is straightforward to see that the number of ways to label the vertical steps of a Dyck path $D \in \mathcal{D}_{\frac{m}{n}}$ to make a semistandard parking function of rank $r$ is ${\binom{r + a_{1} - 1}{a_{1}} \times \cdots \times \binom{r + a_{\ell} - 1}{a_{\ell}}}$, where $a_{1}, \dots, a_{\ell}$ are the lengths of the vertical runs of $D$. The result now follows from~\eqref{eq:comb character}.
\end{proof}

\begin{Ex}
Let us consider the example $m = 3$, $n = r = 2$. There are $C_{\frac{3}{2},2} = 10$ $\frac{3}{2}$-semistandard parking functions of rank $2$, given in Figure~\ref{figure:ssparking}.
\begin{figure}[h]
\begin{tikzpicture}[scale=0.5]

\draw(-4,0)--(-2,0)--(-2,3)--(-4,3)--cycle;
\draw(-4,1)--(-2,1);
\draw(-4,2)--(-2,2);
\draw(-3,0)--(-3,3);
\draw[dashed, color=red](-4,0)--(-2,3);
\draw[line width = 2](-4,0)--(-4,3)--(-2,3);
\node at (-4.3, 0.5) {1};
\node at (-4.3, 1.5) {1};
\node at (-4.3, 2.5) {1};

\draw(0,0)--(2,0)--(2,3)--(0,3)--cycle;
\draw(0,1)--(2,1);
\draw(0,2)--(2,2);
\draw(1,0)--(1,3);
\draw[dashed, color=red](0,0)--(2,3);
\draw[line width = 2](0,0)--(0,3)--(2,3);
\node at (-0.3, 0.5) {2};
\node at (-0.3, 1.5) {1};
\node at (-0.3, 2.5) {1};

\draw(4,0)--(6,0)--(6,3)--(4,3)--cycle;
\draw(4,1)--(6,1);
\draw(4,2)--(6,2);
\draw(5,0)--(5,3);
\draw[dashed, color=red](4,0)--(6,3);
\draw[line width = 2](4,0)--(4,3)--(6,3);
\node at (3.7, 0.5) {2};
\node at (3.7, 1.5) {2};
\node at (3.7, 2.5) {1};

\draw(8,0)--(10,0)--(10,3)--(8,3)--cycle;
\draw(8,1)--(10,1);
\draw(8,2)--(10,2);
\draw(9,0)--(9,3);
\draw[dashed, color=red](8,0)--(10,3);
\draw[line width = 2](8,0)--(8,3)--(10,3);
\node at (7.7, 0.5) {2};
\node at (7.7, 1.5) {2};
\node at (7.7, 2.5) {2};

\draw(-8, -4)--(-6, -4) -- (-6, -1) -- (-8, -1) -- cycle;
\draw(-8, -3) -- (-6, -3);
\draw(-8, -2) -- (-6, -2);
\draw(-7, -4) -- (-7, -1);
\draw[dashed, color=red](-8, -4)--(-6, -1);
\draw[line width = 2](-8, -4)--(-8, -2) -- (-7, -2) -- (-7, -1) -- (-6, -1);
\node at (-8.3, -3.5) {1};
\node at (-8.3, -2.5) {1};
\node at (-8.3, -1.5) {1};

\draw(-4, -4)--(-2, -4) -- (-2, -1) -- (-4, -1) -- cycle;
\draw(-4, -3) -- (-2, -3);
\draw(-4, -2) -- (-2, -2);
\draw(-3, -4) -- (-3, -1);
\draw[dashed, color=red](-4, -4)--(-2, -1);
\draw[line width = 2](-4, -4)--(-4, -2) -- (-3, -2) -- (-3, -1) -- (-2, -1);
\node at (-4.3, -3.5) {1};
\node at (-4.3, -2.5) {1};
\node at (-4.3, -1.5) {2};

\draw(0, -4)--(2, -4) -- (2, -1) -- (0, -1) -- cycle;
\draw(0, -3) -- (2, -3);
\draw(0, -2) -- (2, -2);
\draw(1, -4) -- (1, -1);
\draw[dashed, color=red](0, -4)--(2, -1);
\draw[line width = 2](0, -4)--(0, -2) -- (1, -2) -- (1, -1) -- (2, -1);
\node at (-0.3, -3.5) {2};
\node at (-0.3, -2.5) {1};
\node at (-0.3, -1.5) {1};

\draw(4, -4)--(6, -4) -- (6, -1) -- (4, -1) -- cycle;
\draw(4, -3) -- (6, -3);
\draw(4, -2) -- (6, -2);
\draw(5, -4) -- (5, -1);
\draw[dashed, color=red](4, -4)--(6, -1);
\draw[line width = 2](4, -4)--(4, -2) -- (5, -2) -- (5, -1) -- (6, -1);
\node at (3.7, -3.5) {2};
\node at (3.7, -2.5) {1};
\node at (3.7, -1.5) {2};

\draw(8, -4)--(10, -4) -- (10, -1) -- (8, -1) -- cycle;
\draw(8, -3) -- (10, -3);
\draw(8, -2) -- (10, -2);
\draw(9, -4) -- (9, -1);
\draw[dashed, color=red](8, -4)--(10, -1);
\draw[line width = 2](8, -4)--(8, -2) -- (9, -2) -- (9, -1) -- (10, -1);
\node at (7.7, -3.5) {2};
\node at (7.7, -2.5) {2};
\node at (7.7, -1.5) {1};

\draw(12, -4)--(14, -4) -- (14, -1) -- (12, -1) -- cycle;
\draw(12, -3) -- (14, -3);
\draw(12, -2) -- (14, -2);
\draw(13, -4) -- (13, -1);
\draw[dashed, color=red](12, -4)--(14, -1);
\draw[line width = 2](12, -4)--(12, -2) -- (13, -2) -- (13, -1) -- (14, -1);
\node at (11.7, -3.5) {2};
\node at (11.7, -2.5) {2};
\node at (11.7, -1.5) {2};
\end{tikzpicture}
\caption{The $\frac{3}{2}$-semistandard parking functions of rank $2$. Note that, if $T_{0}$ denotes a maximal torus of $\on{GL}_{2}$, then the $T_{0}$-character of $L_{\frac{3}{2}, 2}$ is given by $2q_{1}^{-3} + 3q_{1}^{-2}q_{2}^{-1} + 3q_{1}^{-1}q_{2}^{-2} + 2q_{2}^{-3}$}
\label{figure:ssparking}
\end{figure}
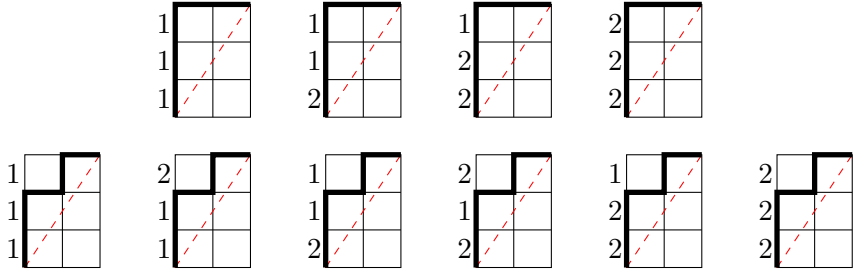
\end{Ex}

\rem{}\label{rmk:T0 from parking}
Recall that $T_{0} \subseteq \on{GL}_{r}$ denotes a maximal torus. It follows from~\eqref{eq:comb character} that the $T_{0}$-character of $\overline{L}_{\frac{m}{n}, r}$ is given by
$$
\sum_{(D, \varphi) \in \mathcal{PF}^{r}_{\frac{m}{n}}} \prod_{i = 1}^{r}q_{i}^{-|\varphi^{-1}(i)|},
$$

\noindent see also Remark \ref{rmk:generating function} above.
\erem

\rem{}
It is an interesting question to find $|\mathcal{PF}^{r}_{\frac{m}{n}}|$ in the non-coprime case. It is possible that a Bizley-like formula exists for the generating function of $|\mathcal{PF}^{r}_{\frac{dm}{dn}}|$, but we will not pursue it here.
\erem

\ssec{}{$q$-analogues of $C_{\frac{m}{n}, r}$} Let us compute $\on{ch}_{\BC^{\times} \times \on{GL}_{r}}\overline{L}_{\frac{m}{n}, r}$ in a couple of easy examples. 

\begin{Ex}
Let us consider the case $n = 3,\, m = 2$. Then
$$
\on{ch}_{\BC^{\times} \times \on{GL}_{r}}(\overline{L}_{\frac{2}{3}, r}) = (q+q^{-1})[S^{2}(\BC^{r*})] + [\Lambda^2(\BC^{r*})], 
$$

\noindent on the other hand, if $n = 2,\, m = 3$ we have
$$
\on{ch}_{\BC^{\times} \times \on{GL}_{r}}(\overline{L}_{\frac{3}{2}, r}) = (q + q^{-1})[S^{3}(\BC^{r*})] + [W_{r}(2,1)^{*}].
$$
\end{Ex}

Note that the roots of $\on{ch}_{\BC^{\times}}(\overline{L}_{\frac{2}{3}, r}) = \on{Tr}_{\overline{L}_{\frac{m}{n}, r}}(q^{\mathbf{h}}) =  \frac{r}{2}((r+1)q + (r-1) + (r+1)q^{-1})$ are roots of unity if and only if $r = 1$. It follows, in particular, that $\on{ch}_{\BC^{\times}}(\overline{L}_{\frac{m}{n}, r})$ does not admit an expression involving only products and quotients of $q$-numbers.

To remedy this, we propose an alternative evaluation of the trace that yields an expression that does factor. In fact, we have several of them, one for each divisor $d$ of $r$. Let us adopt the notation of Remark \ref{rmk:generating function}, in particular, ${g(q, q_{1}, \dots, q_{r}) := \on{Tr}_{\overline{L}_{\frac{m}{n}, r}}(q^{\mathbf{h}}q_{1}^{E_{11}}\cdots q_{r}^{E_{rr}})}$. Let $d$ be a divisor of $r$, and set $k := r/d$. We consider the expression
$$
\frac{[nr]_{\bold{q}}}{[n]_{q}} = [k]_{q^{n}}[d]_{\bold{q}},
$$

\noindent where we set $\bold{q} := q^{\frac{1}{d}}$. Clearly, this is a Laurent polynomial in $\bold{q}$ with non-negative integer coefficients. Now let $N_{d} \in T_{0} \subseteq \on{SL}_{r}$ be 
$$
N_{d} = \on{diag}(q_{1}, \dots, q_{r}),
$$

\noindent where $\{q_{1}, \dots, q_{r}\}$ $= \{\bold{q}^{\frac{dn(k+1 - 2\ell) + d + 1 - 2s}{2}} \,|\, 1 \leqslant \ell \leqslant k, \, 1 \leqslant s \leqslant d\}.$ Note that we have ${\on{Tr}(N_{d}) = [k]_{q^{n}}[d]_{\bold{q}}}$. We define
$$
\on{ch}^{d}_{q}(\overline{L}_{\frac{m}{n}, r}) := \on{Tr}_{\overline{L}_{\frac{m}{n}, r}}(q^{\mathbf{h}}N_{d}).
$$

\noindent Equivalently,
$\on{ch}^{d}_{q}(\overline{L}_{\frac{m}{n}, r}) = g(q, q_{1}, \dots, q_{r})$.

\prop{}\label{prop:qdimension}
We have
$$
\on{ch}_{q}^{d}(\overline{L}_{\frac{m}{n}, r}) = \frac{[d]_{\bold{q}}}{[nd]_{\bold{q}}}
\rParB{\begin{matrix}nr + m - 1 \\ m  \end{matrix}}_{\bold{q}}.
$$
\eprop
\begin{proof}
Note that, by \eqref{form_char}, $\on{ch}_{q}^{d}(\overline{L}_{\frac{m}{n}, r})$ is 
$$
\on{ch}_{q}^{d}(\overline{L}_{\frac{m}{n}, r}) = \frac{1}{[n]_{q}}\on{Tr}(S^{m}(q^{\rho_{n}} \otimes N_{d}^{-1})),
$$

\noindent where $q^{\rho_{n}} = \on{diag}(q^{\frac{-n+1}{2}}, q^{\frac{-n+3}{2}}, \dots, q^{\frac{n-1}{2}})$. The matrix $q^{\rho_{n}}\otimes N_{d}^{-1}$ is diagonal:
$$
q^{\rho_{n}} \otimes N_{d}^{-1} = \on{diag}(q^{\frac{n + 1 - 2i}{2}}(\bold{q}^{-1})^{\frac{dn(k + 1 - 2\ell) + d + 1 - 2s}{2}}),
$$

\noindent where $1 \leqslant \ell \leqslant k$, $1 \leqslant s \leqslant d$ and $1 \leqslant i \leqslant n$. It is easy to see that, up to permuting the diagonal entries, this can be simplified as
$$
q^{\rho_{n}} \otimes N_{d}^{-1} = \on{diag}(\bold{q}^{\frac{-nr + 1}{2}}, \bold{q}^{\frac{-nr + 3}{2}}, \dots, \bold{q}^{\frac{nr - 1}{2}}) = \bold{q}^{\rho_{nr}}
$$

\noindent and it is well-known, and easy to show, that
$$
\on{Tr}S^{m}(\bold{q}^{\rho_{nr}}) =\rParB{\begin{matrix}nr + m - 1 \\ m  \end{matrix}}_{\bold{q}},
$$

\noindent the result now follows from the observation that $[n]_{q} = [nd]_{\bold{q}}/[d]_{\bold{q}}$.
\end{proof}

Thanks to Proposition \ref{prop:qdimension} if $d$ is a divisor of $r$ the $q$-number
$$
C^{d}_{\frac{m}{n}, r}(q) := \frac{[d]_{q}}{[nd]_{q}}\left[\begin{matrix}nr + m - 1 \\ m \end{matrix}\right]_{q}
$$

\noindent is a Laurent polynomial in $q$ with non-negative integer coefficients. Clearly, when $q = 1$ we recover the rank $r$ Catalan number $C_{\frac{m}{n}, r}$. When $r = 1 = d$, the $q$-number $C_{\frac{m}{n}, 1}^{1}(q)$ coincides with the usual $q$-Catalan number, that is the generating function for the $\on{area} - \on{dinv}$ statistic on the set of $\frac{m}{n}$-Dyck paths. We do not know the combinatorial meaning of $C^{d}_{\frac{m}{n}, r}(q)$ for $r > 1$.

\rem{}
More generally, it would be interesting to find statistics on the set of $\frac{m}{n}$-semistandard parking functions of rank $r$ whose generating function is $\on{ch}_{\BC^{\times} \times \on{GL}_{r}}(\overline{L}_{\frac{m}{n}, r}) \in \BZ_{\geq 0}[q, q^{-1}, q_{1}, q_{1}^{-1}, \dots, q_{r}, q_{r}^{-1}]$. The statistics corresponding to $q_{1}, \dots, q_{r}$ are not hard, see Remark \ref{rmk:T0 from parking} above. One possible way to find a statistic corresponding to $q$ is to introduce analogues of sweep maps for semistandard parking functions, see Section 6.1 in \cite{ALW}. We plan to come back to this in future work. 
\erem

		\sec{Loc}{Localization}\label{loc_simple_prf}
		\ssec{Loc}{Quantizations of $\mathfrak{M}^\theta(n,r)$ and localization}\label{quant_of_gies}
		Let us now consider quantizations $\CA^{\theta}_c(n,r)$ of $\mathfrak{M}^\theta(n,r)$ which are sheaves in conical topology of filtered algebras on $\mathfrak{M}^\theta(n,r)$ with $\on{gr}\CA^{\theta}_c(n,r) \simeq \CO_{\mathfrak{M}^\theta(n,r)}$. These quantizations are defined via
	\begin{equation*}
	    \CA_c^\theta(n,r)=D_R/\!\!/\!\!/^\theta_c \on{GL}(V):=p_*(D_R/D_R\{\xi_R - c \on{tr}(\xi)\,|\, \xi \in \mathfrak{gl}_n\}|_{\mu^{-1}(0)^{\theta\text{-}\on{st}}})^{\on{GL}_n},
	\end{equation*}
where $D_R$ is the sheaf (in conical topology) of microlocal differential operators on $T^{*}R$ and $p\colon \mu^{-1}(0)^{\theta\text{-}\on{st}} \twoheadrightarrow \mathfrak{M}^\theta(n,r)$ is the quotient morphism. 

\rem{}
Note that one can also define $\CA^\theta_c(n,r)$ in the following way. We need to define its sections over principal open subsets $(T^*R)_f/\!\!/\!\!/^\theta \on{GL}_n \subset \mathfrak{M}^\theta(n,r)$ where $f\in \BC[T^*R]$ is a $\BC^\times$-homogeneous $\theta$-semiinvariant function of degree $\geqslant 1$. Let $\on{R}_{\hbar} := \on{R}_{\hbar}(D(R))$ be the Rees algebra of $D(R)$, so that $\on{R}_{\hbar}/(\hbar) = \BC[T^{*}R]$. For $k > 0$, let $S_{k} \subseteq \on{R}_{\hbar}/(\hbar^{k})$ be the preimage of the set $\{f^{m}\,|\, m \geqslant 0\}$ with respect to the natural morphism $\on{R}_{\hbar}/(\hbar^{k}) \twoheadrightarrow \on{R}_{\hbar}/(\hbar)$. It is easy to see that $S_{k}$ is an Ore set in $\on{R}_{\hbar}/(\hbar^{k})$ and that we can form the inverse limit 

$$
\widehat{\on{R}_{\hbar}}[f^{-1}] := \underset{\longleftarrow}{\on{lim}} \, \left(\on{R}_{\hbar}/(\hbar^{k})\right)[S_{k}^{-1}].
$$

\noindent   We consider the subalgebra of locally finite vectors with respect to the natural $\BC^{\times}$-action that we denote by $\widehat{\on{R}_{\hbar}}[f^{-1}]_{\on{l.f.}}$, and we define

$$D(R)[f^{-1}] := \widehat{\on{R}_{\hbar}}[f^{-1}]_{\on{l.f.}}/(\hbar - 1).$$

\noindent It is straightforward to see that $D(R)[f^{-1}]$ is a filtered quantization of $\BC[T^{*}R][f^{-1}]$. Since $f$ is $\theta$-semiinvariant we still have an action of $\operatorname{GL}_{n}$ on $D(R)[f^{-1}]$ and a quantum comoment map $\mathfrak{gl}_{n} \to D(R)[f^{-1}]$. The sections of $\CA^{\theta}_{c}(n, r)$ on $(T^*R)_f/\!\!/\!\!/^\theta \on{GL}_n$ are then the quantum Hamiltonian reduction $D(R)[f^{-1}]/\!\!/\!\!/_c \on{GL}_n$. All of this follows from the construction of the sheaf $D_{R}$, see \cite[Section 1]{ginsburg} for details. 
\erem


We analogously define quantizations $\overline{\CA}_c^\theta(n,r)$ of $\overline{\mathfrak{M}}^\theta(n,r)$.
We write $\CA^\theta_c(n,r)\text{-mod}$ (resp. $\overline{\CA}^\theta_c(n,r)\text{-mod}$) for the category of coherent $\CA^\theta_c(n,r)$-modules (resp. $\overline{\CA}^\theta_c(n,r)$-modules) and $\CA_c(n,r)\text{-mod}$ (resp. $\overline{\CA}_c(n,r)\text{-mod}$) for the category of finitely generated $\CA_c(n,r)$-modules (resp. $\overline{\CA}_c(n,r)$-modules).
We have the global sections functor $\Gamma^\theta_c\colon \CA^\theta_c(n,r)\text{-mod} \ra \CA_c(n,r)\text{-mod}$ (resp. $\overline{\Gamma}^\theta_c\colon \overline{\CA}^\theta_c(n,r)\text{-mod} \ra \overline{\CA}_c(n,r)\text{-mod}$). 
We say that the abelian localization holds for $(\theta,c)$ if the functor $\overline{\Gamma}^{\theta}_c$ (or equivalently $\Gamma^\theta_c$) is an abelian equivalence.

In this section, using the fact that $\overline{\CA}_{\frac{m}{n}}(n,r)$ admits a finite-dimensional representation, we will simplify the proof of the following theorem (see~\cite[Theorem~1.1 (2)]{L}).

\th{Loc}\label{loc}
For $\theta>0$ (resp. $\theta<0$), abelian localization holds for $c \in \BC$ iff $c$ is not of the form $\frac{s}{m}$, where $1 \leqslant m \leqslant n$ and $s<0$ (resp. if $c$ is not of the form $-r-\frac{s}{m}$, where $1 \leqslant m \leqslant n$ and $s<0$).
\eth

Recall that the proof of this theorem in~\cite[Section 5]{L} consists of three steps (see the beginning of~\cite[Section 5]{L}). In the first step the proof of Theorem~\ref{loc} reduces to the case when $c=\frac{m}{n'} \geqslant 0$ with $m,\,n'$ coprime, $n' \leqslant n$ and $\theta>0$. In the second step the proof reduces to the case $n=n'$ and $c>0,\,\theta > 0$. 
In the third step the claim reduces to the fact that the functor $\Gamma^\theta_c$ induces an equivalence between certain categories $\CO_\nu(\CA^\theta_c(n,r)),\, \CO_\nu(\CA_c(n,r))$ over $\CA^\theta_c(n,r),\, \CA_c(n,r)$ (see Section~\ref{cat_O_supp} for the definitions of these categories). The last claim is proved via proving that the number of simples of the categories $\CO_\nu(\CA^\theta_c(n,r)),\, \CO_\nu(\CA_c(n,r))$ are equal. 
The last step is crucial and we will simplify its proof.
So, from now on we assume that $c=\frac{m}{n}>0,\, \on{gcd}(m,n)=1$ and $\theta>0$.

Let us first of all define categories $\mathcal{O}$ and other notions and objects that we will use in the proof. We use the same notations as in~\cite{L}.

		\ssec{Loc}{Singular support and categories $\CO$}\label{cat_O_supp} 
		 Let $\nu\colon \BC^\times \ra T=\BC^\times \times T_0$ be a cocharacter given by  $t \mapsto (t,t^{d_1},\ldots,t^{d_r})$ for $d_1 \gg \ldots \gg d_r$. We will denote by $\nu_0$ the cocharacter of $T$ given by $t \mapsto (1,t^{d_1},\ldots,t^{d_r})$. 
		The cocharacter $\nu$ (resp. $\nu_0$)  induces a grading $\CA_c(n,r)=\bigoplus_i \CA_c^{i,\nu}$ (resp. $\CA_c(n,r)=\bigoplus_i \CA_c^{i,\nu_0}$). 
		We set $\CA_c^{>0,\nu}:=\bigoplus_{i>0} \CA_c^{i,\nu}$.
		The action of $\nu$ 
    	on $\CA_c(n,r)$ is Hamiltonian, let $h \in \CA^{0,\nu}_c$ be the image of $1$ under the comoment map. The grading $\CA_c(n,r) = \bigoplus_i \CA_c^{i, \nu}$ is inner and is given by $h$. 
		Define the category $\CO_{\nu}(\CA_c(n,r))$ as the full subcategory of the category $\CA_c(n,r)$-mod consisting of all modules where the action of $\CA_c^{>0}(n,r)$ is locally nilpotent.
Let us also define the category $\CO_\nu(\CA^\theta_c(n,r))$ as the full subcategory of $\CA^\theta_c(n,r)\text{-mod}$ consisting of modules that come with a good filtration stable under $h$ 
and that are supported on the contracting locus 
$\mathfrak{M}^{\theta}_{+}(n,r)$ of $\nu$ in $\mathfrak{M}^\theta(n,r)$. Recall that this contracting locus is defined by
$$\mathfrak{M}^\theta_+(n,r)=\{x\in \mathfrak{M}^\theta(n,r)\,|\,\on{lim}_{t\ra 0}\nu(t)\cdot x~\on{exists}\}.$$

		We set 
		\begin{equation*}
		\mathsf{C}_{\nu_0}(\CA_c(n,r)):=\CA^{0,\nu_0}_c(n,r)/\sum_{i>0}\CA^{-i,\nu_0}_c(n,r)\CA^{i,\nu_0}_c(n,r).
		\end{equation*}
		
	\rem{} Note that we have the natural isomorphisms  
	\begin{equation*} \mathsf{C}_{\nu_0}(\CA_c(n,r)) \iso \CA^{\geqslant 0,\nu_0}_c(n,r)/(\CA^{\geqslant 0,\nu_0}_c(n,r) \cap \CA_c(n,r)\CA^{> 0,\nu_0}_c(n,r)),
	\end{equation*} 
	\begin{equation*}
	\mathsf{C}_{\nu_0}(\CA_c(n,r)) \iso 
	\CA^{\leqslant 0,\nu_0}_c(n,r)/(\CA^{\leqslant 0,\nu_0}_c(n,r) \cap \CA^{< 0,\nu_0}_c(n,r)\CA_c(n,r)).
	\end{equation*}
	\erem	
The algebra $\mathsf{C}_{\nu_0}(\CA_c(n,r))$ will be called the {\em{Cartan subquotient}} of $\CA_c(n,r)$ with respect to $\nu_0$. 
One can also define Cartan subquotients of sheaves $\CA^{\theta}_c(n,r)$: it follows from~\cite[Proposition~5.2]{Cat_O_quant} that there exists a unique sheaf $\mathsf{C}_{\nu_0}(\CA^{\theta}_c(n,r))$ in the conical topology on ${\mathfrak{M}^\theta(n,r)}^{\nu_0(\BC^\times)}$ such that for any $\BC^\times \times \nu_0(\BC^\times)$-stable open subvariety ${U \subset \mathfrak{M}^\theta(n,r)}$ with $U^{\nu_0(\BC^\times)} \neq \varnothing$ we have $\mathsf{C}_{\nu_0}(\CA^{\theta}_c(n,r))(U^{\nu_0(\BC^\times)})
=\mathsf{C}_{\nu_0}(\CA_c^\theta(U))$.

		\prop{Loc}\label{fixed_points} For $\theta>0$ we have 
\begin{enumerate}
\item
		$
		\mathfrak{M}^{\theta}(n,r)^{\nu_0(\BC^\times)} = \bigsqcup_{n_1+\ldots+n_r=n} \prod_i \mathfrak{M}^{\theta}(n_i,1)
		$,
		where the disjoint union runs over all ordered collections $(n_{1}, \ldots, n_{r})$ of non-negative integers such that ${n_{1} + \ldots + n_{r} = n}$ (and we set $\mathfrak{M}^{\theta}(0, 1) = \on{pt}$).
\item  For the connected component $Z \subset \mathfrak{M}^{\theta}(n,r)^{\nu_0(\BC^\times)}$ which corresponds to the composition $(n_1,\ldots,n_r)$ of $n$ we have ${\mathsf{C}_{\nu_0}(\CA^{\theta}_c(n,r))}|_Z = \bigotimes_i \CA^{\theta}_{c+r-i}(n_i,1)$, where we set $\CA^{\theta}_{c}(0, 1) = \BC$ (a sheaf on $\mathfrak{M}^{\theta}(0, 1) = \on{pt}$).
\item  For a Zariski generic $c \in \BC$ we have $\mathsf{C}_{\nu_0}(\CA_c(n,r))=\bigoplus\bigotimes_i \CA_{c+r-i}(n_i,1)$, where the direct sum is taken over all ordered collections $(n_{1}, \dots, n_{r})$ as in $(1)$. 
\end{enumerate}	
		\eprop
		
		\prf
		Follows from~\cite[Proposition~3.5]{L}. The proof should be changed as follows (see the footnote $1$ above): note that the line bundle $\CO(1)$ on $\mathfrak{M}^\theta(n_i,1)$ is 
		the top exterior power of the bundle on $\mathfrak{M}^\theta(n_i,1)$ induced from the representation $\on{GL}_{n_i} \curvearrowright {\BC^{n_i}}^*$, so we conclude that $c_1(Y_\mu)=\sum_{i=1}^r(2i-r-1)c_i$ (not $\sum_{i=1}^r(r+1-2i)c_i$), i.e. the period of $\mathsf{C}_{\nu_0}(\CA^\theta_\la)$ equals  
		$\sum_{i=1}^r(\la+r+\frac{1}{2}-i)$, hence, 
		$\mathsf{C}_{\nu_0}(\CA^\theta_\la)|_{Z}\simeq \bigotimes_i \CA^\theta_{c+r-i}(n_i,1)$.
		\epr

		Let us now give an important property of the category $\mathcal{O}(\CA^{\theta}_{c}(n, r))$. 
		
	\th{}
	The category $\mathcal{O}_{\nu}(\CA^{\theta}_{c}(n, r))$ is a highest weight category, with simples indexed by the fixed points of $\nu$ and the order in the definition of a highest weight category is the contraction order on the fixed points.
	\eth
	\begin{proof}
	Since the action of $\nu$ on $\mathfrak{M}^{\theta}(n, r)$ has finitely many fixed points, this follows from \cite[Proposition 5.17]{blpw}.
	\end{proof}

 Starting from a module $M \in \CA_c(n,r)\text{-mod}$ (resp. $\CM \in \CA^\theta_c(n,r)\text{-mod}$) we can construct its {\em{associated
variety}}, to be denoted $V(M) \subset \mathfrak{M}(n,r)$ (resp. $V(\CM) \subset \mathfrak{M}^\theta(n,r)$), as follows. Consider any good filtration $F^\bullet M$ (resp. $F^\bullet \CM$), then define $V(M)$ (resp. $V(\CM)$) as the support of $\on{gr}F^\bullet M$ (resp. $\on{gr}F^\bullet \CM$) with the reduced scheme structure. It is straightforward, and well-known, that this does not depend on the choice of a  good filtration. Moreover, $V(\CM)$ is simply the support of $\CM$ considered as a sheaf on $\mathfrak{M}^{\theta}(n, r)$.

		\prop{Loc}\label{Bern_ineq}
		Let $L$ be a simple module of the category $\CO_\nu(\overline{\CA}_c(n,r))$ and let $\mathcal{I} \subset \overline{\CA}_c(n,r)$ be the annihilator of $L$. Then 
		$2\on{dim}(V(L))=\on{dim}(V(\overline{\CA}_c(n,r)/\mathcal{I}))$, where $V(\overline{\CA}_{c}(n, r)/\mathcal{I})$ is computed by considering $\overline{\CA}_{c}(n,r)/\mathcal{I}$ as a left $\overline{\CA}_{c}(n,r)$-module.
		\eprop
		
\prf
Note that by~\cite[Section 4.4]{Cat_O_quant} every module of the category $\CO_\nu(\overline{\CA}_c(n,r))$ is holonomic in the sense of~\cite{Lo1}.
Now the claim follows from~\cite[Theorems~1.2,\,1.3]{Lo1}. 
\epr
		
\ssec{Loc}{Properties of $\Gamma^\theta_c$ and categories $\mathcal{O}$}

Let us now study the interaction between the global sections functor and the categories $\mathcal{O}$.

		\prop{Loc}\label{quot}\label{Loc_sec} Assume that $\theta > 0$ and $c>-r$. Then the following holds.
		\begin{enumerate}
		\item
		 The functor $\Gamma^{\theta}_c\colon \CO_\nu(\CA^{\theta}_c(n,r))) \twoheadrightarrow \CO_\nu(\CA_c(n,r)))$ is a quotient functor. 
		\item
		The canonical adjunction morphism  $M \to \Gamma^{\theta}_c \circ \on{Loc}(M)$ is an isomorphism for any $M \in \CO_\nu(\CA_c(n,r)))$.
		\end{enumerate}
		\eprop
		
		\prf
		Part $(1)$ follows from~\cite[Section~8]{mn} and~\cite[Proposition~5.1]{L}. 	Part $(2)$ follows from the fact that $\on{Loc}$ is left adjoint to $\Gamma^{\theta}_c$ and
		general properties of quotient functors between finite categories.
		\epr
		
		Let us now return to the proof of Theorem~\ref{loc}. Recall that we can assume that $c=\frac{m}{n}>0,\, \on{gcd}(m,n)=1,\, \theta>0$.
        It now follows from~\cite[Steps 2,\,3 of Proposition~5.6]{L}
       that to finish the proof of Theorem~\ref{loc} it is enough to prove the following theorem.
		
		\th{}\label{thm:O}
		Assume $\theta > 0$ and $c = \frac{m}{n} > 0$ with $\gcd(m,n) = 1$. Then $\overline{\Gamma}^{\theta}_{c}\colon \mathcal{O}_{\nu}(\overline{\CA}^{\theta}_{c}(n, r)) \to \mathcal{O}_{\nu}(\overline{\CA}_{c}(n, r))$ is an equivalence of categories.
		\eth
		
		Thanks to Proposition~\ref{quot}, to prove Theorem~\ref{thm:O} (and therefore also Theorem~\ref{loc}) it is enough to show that $\overline{\Gamma}^{\theta}_{c}(\CL)\neq 0$ for every simple $\CL \in \CO_\nu(\overline{\CA}^\theta_c(n,r))$. We consider two cases, according to the associated variety of simples.

\ssec{}{Associated varieties of simples and proof of Theorem~\ref{thm:O}} 
First we show that in the situation of Theorem~\ref{thm:O} if $\CL \in \CO(\overline{\CA}_c^\theta(n,r))$ is such that $V(\CL) \nsubseteq \overline{\rho}^{-1}(0)$, then $\overline{\Gamma}^\theta_c(\CL)$ is infinite dimensional so in particular nonzero.
		
		\lem{Loc}\label{supp_infdim}
		Assume that $c=\frac{m}{n}>0$ with $m,\,n$ coprime. 
		For any infinite dimensional simple module $L \in \CO(\overline{\CA}_c(n,r))$ we have $V(\overline{\CA}_c(n,r)/\mathcal{I})=\overline{\mathfrak{M}}(n,r)$, where $\mathcal{I} \subset \overline{\CA}_c(n,r)$ is the annihilator of $L$. 
		\elem
		
		\prf
		Follows from the proof of Step $3$ in the proof of Proposition $4.1$ in~\cite{L}.
		\epr
		
\rem{}	
{We keep the notations of Lemma~\ref{supp_infdim}. Note that by~\cite[Theorem~1.3]{L} $\CA_c$ is a prime ring and $\mathcal{I}$ is primitive, hence, the equality $V(\overline{\CA}_c(n,r)/\mathcal{I})=\overline{\mathfrak{M}}(n,r)$ implies $\mathcal{I}=0$. We are grateful to the anonymous referee for pointing this out to us.}
\erem

The variety $\ol{\mathfrak{M}}(n, r)$ is Poisson and we denote by $\ol{\mathfrak{M}}(n, r)^{\textrm{reg}}$ its unique open symplectic leaf. The map $\ol{\rho}: \ol{\mathfrak{M}}^{\theta}(n, r) \to \ol{\mathfrak{M}}(n, r)$ is an isomorphism over $\ol{\mathfrak{M}}(n, r)^{\textrm{reg}}$ and we denote by $\ol{\mathfrak{M}}^{\theta}(n, r)^{\textrm{reg}}$ the preimage of $\ol{\mathfrak{M}}(n, r)^{\textrm{reg}}$. 
		
		\lem{Loc}\label{inter_reg}
		Assume that $c=\frac{m}{n}>0$ with $m,\,n$ coprime. 
		Let $\CL \in \CO(\overline{\CA}^{\theta}_c(n,r))$ be an irreducible sheaf such that $V(\CL) \nsubseteq \overline{\rho}^{-1}(0)$, then $V(\CL) \cap \ol{\mathfrak{M}}^\theta(n,r)^{\on{reg}} \neq \varnothing$.
		\elem
		
		\prf
		By~\cite[Theorem~A]{BPW} there exists $N \gg 0$ such that the abelian localization holds for $(c+N\theta,\theta)$. Note that the categories $\CO(\overline{\CA}^\theta_{c}(n,r)), \CO(\overline{\CA}^\theta_{c+N\theta}(n,r))$ are equivalent via translation functors and these functors preserve associated varieties. Note also that $c+N\theta$ is obviously of the form $\frac{m'}{n}$ with $m',\,n$ coprime. 
		So, applying translation functors if needed, we can assume abelian localization holds. So $\CL=\on{Loc}(L)$ for some irreducible object $L \in \CO(\overline{\CA}_c(n,r))$ such that $V(L) \nsubseteq \{0\}$. It remains to show that $V(L) \cap \overline{\mathfrak{M}}(n,r)^{\textrm{reg}} \neq \varnothing$. Let $\mathcal{I} \subset \overline{\CA}$ be the annihilator of $L$. It follows from Lemma~\ref{supp_infdim} that $V(\overline{\CA}/\mathcal{I})=\overline{\mathfrak{M}}(n,r)$. It now follows from~\cite[Theorem~1.1]{Lo1} that $V(L)$ has nonempty intersection with the open symplectic leaf of $\overline{\mathfrak{M}}(n,r)$, which is exactly $\overline{\mathfrak{M}}(n,r)^{\textrm{reg}}$.
		\epr
		
		\cor{Loc}\label{global_big_supp}
		Assume that $c=\frac{m}{n}>0$ with $m,\,n$ coprime. 
		Let $\CL \in \CO(\overline{\CA}_c^\theta(n,r))$ be an irreducible sheaf such that $V(\CL) \nsubseteq \overline{\rho}^{-1}(0)$. Then $\overline{\Gamma}^\theta_c(\CL)$ is infinite dimensional.
		\ecor
		
		\prf
        It follows from Lemma~\ref{inter_reg} that we have a point $x \in V(\CL) \cap \ol{\mathfrak{M}}^\theta(n,r)^{\textrm{reg}}$. Note now that the morphism $\overline{\rho}$ is an isomorphism over the open subvariety ${\overline{\mathfrak{M}}^\theta(n,r)^{\textrm{reg}} \subset \overline{\mathfrak{M}}^\theta(n,r)}$, $\overline{\rho}\colon \overline{\mathfrak{M}}^\theta(n,r)^{\textrm{reg}} \iso \overline{\mathfrak{M}}(n,r)^{\textrm{reg}}$. It follows that $\overline{\rho}(x) \in V(\overline{\Gamma}^{\theta}_c(\CL))$,
        hence, 	$\overline{\Gamma}^{\theta}_c(\CL)$ is infinite dimensional. 
		\epr
		
Let us now deal with the case $V(\mathcal{L}) \subseteq \overline{\rho}^{-1}(0)$. 
		
				\prop{Loc}\label{fd_sh}
 		The number of simple coherent $\overline{\CA}^{\theta}_c(n,r)$-modules supported on $\overline{\rho}^{-1}(0)$ cannot be bigger then $1$.
 		\eprop
 		
 		\prf
 		This follows from Step $5$ of the proof of Proposition $4.1$ in~\cite{L}.
 		 \epr
 		
 		If a simple coherent $\overline{\mathcal{A}}^{\theta}_{c}(n, r)$-module supported on $\overline{\rho}^{-1}(0)$ exists, we will denote it by $\CL^{\mathrm{fin}}$.\\
 		
 		We are now ready to prove Theorem~\ref{thm:O}.
 		
 		\begin{proof}[Proof of Theorem~\ref{thm:O}] 
 		From Proposition~\ref{quot} it follows that it is enough to show that for any simple $\CL \in \CO(\overline{\CA}^{\theta}_c(n,r))$ we have $\overline{\Gamma}^\theta_c(\CL)\neq 0$. If $\CL \not\simeq \CL^{\mathrm{fin}}$, then the desired result follows from Corollary~\ref{global_big_supp}. It also follows from Corollary~\ref{global_big_supp} that for any such $\CL$ we have $\overline{\Gamma}^{\theta}_{c}(\CL) \not\simeq \overline{L}_{\frac{m}{n},r}$. 
		The functor $\overline{\Gamma}^\theta_c$ is a quotient functor and $\overline{\Gamma}^{\theta}_{c}(\CL) \not\simeq \overline{L}_{\frac{m}{n},r}$ for any $\CL  \not\simeq \CL^{\mathrm{fin}}$, hence, $\overline{\Gamma}^{\theta}_{c}(\CL^{\mathrm{fin}}) \simeq \overline{L}_{\frac{m}{n},r} \neq 0$, where the last inequality follows from Remark~\ref{zero_dim}.
 		\end{proof}

\section{Representations with minimal support}\label{repres_min_supp}
\ssec{}{Construction of minimally supported modules} Now we generalize results on finite-dimensional representations to representations with minimal support (see Section~\ref{cat_O_supp}). We continue studying the algebra $\overline{\CA}_{\frac{m}{n}}(n, r)$, and set $d := \gcd(m,n)$. The difference now is that we do not assume $d = 1$. Let $m_{0} := \frac{m}{d}$, $n_{0} := \frac{n}{d}$. Let $\lambda$ be a partition of $d$ and consider the partition $n_{0}\lambda$ of $n$. We will denote by $M(O(n_{0}\lambda))$ the irreducible $\chi$-equivariant $D$-module on $\mathfrak{sl}_{n}$ associated to the nilpotent orbit $O(n_{0}\lambda)$ where, recall, $\chi = \frac{m}{n}\on{tr}$. 

\lem{}\label{lemma:highest weight}
We have $F_{n, m, 1}(M(O(n_{0}\lambda))) = S_{\frac{n}{m}}(m_{0}\lambda)$, where $S_{\frac{n}{m}}(m_{0}\lambda)$ is the irreducible highest-weight representation of $H_{\frac{n}{m}}(m, 1)$ with highest weight $m_{0}\lambda$.
\elem
\begin{proof}
Note that a similar result is proven in \cite[Theorem 9.12]{e} for $F^{*}_{n, m, 1}(M(O(n_{0}\lambda)))$, where \emph{highest} weight is replaced by \emph{lowest} weight, see Remark \ref{rmk:Fourier}. The same proof applies, {\it mutatis mutandis}, in our situation. More precisely, the functions constructed in \cite[Lemma 9.13]{e} are annihilated by $x_{1}, \dots, x_{m}$ and span a copy of the representation of $S_{m}$ indexed by $m_{0}\lambda$, see \cite[Lemma 9.15]{e}.
\end{proof}

It follows that we have a $q$-graded $\on{GL}_{r}$-equivariant isomorphism
\begin{equation}\label{isomo}
\overline{L}_{\frac{m}{n}, r}(n_{0}\lambda) := (M(O(n_{0}\lambda)) \otimes \BC[\Hom(\BC^n, \BC^r)])^{\on{GL}_{n}} \iso (S_{\frac{n}{m}}(m_{0}\lambda) \otimes (\BC^{r*})^{\otimes m})^{S_{m}}.
\end{equation}

\noindent From here, we can read the character of $\overline{L}_{\frac{m}{n}, r}(n_{0}\lambda)$, we will do this explicitly in Section~\ref{char_min_supp}. Note that $\overline{L}_{\frac{m}{n},r}(n_{0}\lambda)$ is an irreducible $\overline{\CA}_{\frac{m}{n}}(n,r)$-module. We claim that it has minimal support. Recall from the introduction that, provided ${\overline{L}_{\frac{m}{n}, r}(n_{0}\la) \in \CO_{\nu}(\overline{\CA}_{\frac{m}{n}}(n,r))}$, this means that the GK dimension of $\overline{L}_{\frac{m}{n}, r}(n_{0}\lambda)$ is precisely $d-1$, where, as above, $d = \gcd(m,n)$. So we start by showing that $\overline{L}_{\frac{m}{n}, r}(n_{0}\la) \in \CO_{\nu}(\overline{\CA}_{\frac{m}{n}}(n,r))$.

\prop{}
Let $\nu\colon \BC^\times \ra T$ be a generic co-character given by $t \mapsto (t^k,t^{d_1},\ldots,t^{d_r})$ such that $k>0$. Then $\overline{L}_{\frac{m}{n},r}(n_0\la) \in \CO_{\nu}(\overline{\CA}_{\frac{m}{n}}(n,r))$.
\eprop
\prf
Recall that the action of $\nu$ on $\overline{\CA}_{\frac{m}{n}}(n,r)$ is Hamiltonian. 
Let $h \in \overline{\CA}_{\frac{m}{n}}^{0,\nu}$ be the image of $1$ under the comoment map. To check that $\overline{L}_{\frac{m}{n},r}(n_0\la) \in \CO_{\nu}(\overline{\CA}_{\frac{m}{n}}(n,r))$ it is enough to show that generalized eigenvalues of $h$ acting on $\overline{L}_{\frac{m}{n},r}(n_0\la)$ are bounded from above. This follows from the existence of the $T$-equivariant isomorphism~(\ref{isomo}), the fact that $(\BC^{r*})^{\otimes m}$ is finite dimensional and the fact that $S_{\frac{n}{m}}(m_0\la)$ lies in the category $\CO$ over the Cherednik algebra, see Lemma \ref{lemma:highest weight}. 
\epr

Let us now show that the module $\overline{L}_{\frac{m}{n}, r}(n_{0}\lambda)$ has minimal support. In order to do this, it is enough to compute its GK dimension. 

\prop{}\label{prop:minsupport}
The GK dimension of the $\overline{\CA}_{\frac{m}{n}}(n, r)$-representation $\overline{L}_{\frac{m}{n}, r}(n_{0}\lambda)$ is exactly $d - 1$, where $d = \gcd(m,n)$. In particular, $\overline{L}_{\frac{m}{n}, r}(n_{0}\lambda)$ has minimal support.
\eprop
\begin{proof}
Recall that we have $\overline{L}_{\frac{m}{n}, r}(n_{0}\lambda) = (M(O(n_{0}\lambda)) \otimes \BC[\Hom(\BC^n, \BC^r)])^{\on{GL}_{n}}$. The ${D(\mathfrak{sl}_{n} \oplus \Hom(\BC^n, \BC^r))}$-module $M(O(n_{0}\lambda)) \otimes \BC[\Hom(\BC^n, \BC^r)]$ is holonomic, so it admits a good filtration that induces a good filtration on $\overline{L}_{\frac{m}{n}, r}(n_{0}\lambda)$, both as an $\overline{\CA}_{\frac{m}{n}}(n, r)$-module and as a $H^{\textrm{sph}}_{\frac{n}{m}}(m, r)$-module, where we take the Bernstein filtration on the ring of differential operators $D(\mathfrak{sl}_{n} \oplus \Hom(\BC^n, \BC^r))$. Indeed, that this filtration is good for the $\overline{\CA}_{\frac{m}{n}}(n, r)$-module structure follows by definition, and that it is good for the $H^{\textrm{sph}}_{\frac{n}{m}}(m, r)$-module structure follows from the formulas in Proposition~\ref{prop:operators}.
So it is enough to show that the GK dimension of the $H^{\textrm{sph}}_{\frac{n}{m}}(m, r)$-module ${(S_{\frac{n}{m}}(m_{0}\lambda) \otimes (\BC^r)^{\otimes m})^{S_{m}}}$ is exactly $d-1$. To do this, it is enough to show that the GK dimension of the $H_{\frac{n}{m}}(m, r)$-module $S_{\frac{n}{m}}(m_{0}\lambda) \otimes (\BC^r)^{\otimes m}$ is $d-1$. But thanks to Proposition~\ref{prop:morita}, this coincides with the GK dimension of the $H_{\frac{n}{m}}(m, 1)$-module $S_{\frac{n}{m}}(m_{0}\lambda)$. This is precisely $d-1$ by~\cite[Theorem 1.6]{wi}. We are done. 
\end{proof}

\ssec{}{Coincidence of labels: $L_{\frac{m}{n},r}(n_0\la)=L(\varnothing,\ldots,\varnothing,n_0\la)$}
By Proposition~\ref{prop:minsupport} the module $L_{\frac{m}{n},r}(n_0\la):= \BC[x] \otimes (M(O(n_{0}\lambda)) \otimes \BC[\Hom(\BC^n, \BC^r)])^{\on{GL}(V)}$ is minimally supported so it must have the form $L(\varnothing,\ldots,\varnothing,n_0\la')$ for some partition $\la'$ of $d$. The goal of this section is to show that $\la'=\la$.
We set $G:=\on{GL}(V)$. 
Consider the $D(\overline{R})\text{-}\overline{\CA}_{\frac{m}{n}}(n,r)$-bimodule 
\begin{equation*}
\overline{\CQ}_{\frac{m}{n}}:=D(\overline{R})/(D(\overline{R})\{\xi_{\overline{R}}-\frac{m}{n}\on{tr}\xi\,|\, \xi \in \mathfrak{gl}(V)\}).
\end{equation*}
Recall the character $\on{det}\colon \on{GL}(V) \ra \BC^\times$ and consider the corresponding space of $\on{GL}(V)$-semiinvariants $\overline{\CQ}_{\frac{m}{n}}^{\on{GL}(V),\on{det}}$ which is naturally a $\overline{\CA}_{\frac{m}{n}+1}(n,r)\text{-}\overline{\CA}_{\frac{m}{n}}(n,r)$-bimodule. 
Recall that $\frac{m}{n}>0$, so by Theorem~\ref{loc} the localization holds for $(\frac{m}{n},\on{det}),\, (\frac{m}{n}+1,\on{det})$. It now follows from~\cite[Proposition~5.2]{bl} and~\cite[Proposition~6.31]{BPW}
that the functor
$\overline{\CQ}_{\frac{m}{n}}^{\on{GL}(V),\on{det}}\otimes_{\overline{\CA}_{\frac{m}{n}}(n,r)} \bullet $ induces an equivalence 
\begin{equation*}
\overline{\CT}_{\frac{m}{n} \ra \frac{m}{n}+1}\colon \CO_{\nu}(\overline{\CA}_{\frac{m}{n}}(n,r)) \iso \CO_\nu(\overline{\CA}_{\frac{m}{n}+1}(n,r))
\end{equation*}
which we will call a {\em{translation equivalence}}.
Let us now prove the following lemma.
\lem{}\label{trans_equiv_simp}
Under the equivalence $\overline{\CT}_{\frac{m}{n} \ra \frac{m}{n}+1}$ the module $\overline{L}_{\frac{m}{n},r}(n_0\la)$ maps to $\overline{L}_{\frac{m}{n}+1,r}(n_0\la)$.
\elem
\prf
For $c \in \BC$ we recall that $D(\overline{R})\text{-}\on{mod}^{G,c}$ is the category of $(G,c)$-equivariant $D$-modules on $\overline{R}$.
Note that we have an equivalence 
$$
\widetilde{\CT}_{c \ra c+1}\colon D(\overline{R})\text{-}\on{mod}^{G,c} \iso D(\overline{R})\text{-}\on{mod}^{G,c+1}
$$ 
given by tensoring with a one-dimensional $\on{GL}(V)$-module $\BC_{\on{det}}$ on which $\on{GL}(V)$ acts via $\on{det}$. Now for $c = \frac{m}{n} = \frac{m_{0}d}{n_{0}d}$ we set $\widetilde{L}_c:=M(O(n_0\la)) \otimes \on{Hom}(V,W)$ and consider it as an object of the category $D(\overline{R})\text{-}\on{mod}^{G,c}$. It follows from the definitions that $\widetilde{L}_{c+1}=\BC_{\on{det}}\otimes \widetilde{L}_c=\widetilde{\CT}_{c \ra c+1}(\widetilde{L}_c)$.

Recall that the categories 
$\CO_\nu(\overline{\CA}_{\frac{m}{n}}(n,r)),\, \CO_\nu(\overline{\CA}_{\frac{m}{n}+1}(n,r))$ are highest weight.
It follows from the definitions that the functor $\overline{\CT}_{\frac{m}{n} \ra \frac{m}{n}+1}$ sends $\Delta_{\frac{m}{n}}(p)$ to $\Delta_{\frac{m}{n}+1}(p)$ so it must be label preserving. 
We conclude that to prove lemma it is enough to show that the following diagram is commutative.
 \begin{equation}\label{trans_ham}
\xymatrix{D(\overline{R})\text{-}\on{mod}^{G,\frac{m}{n}} \ar[d]^{\widetilde{\CT}_{\frac{m}{n} \ra \frac{m}{n}+1}} \ar[rr]^{\overline{\pi}_{\frac{m}{n}}} && \ar[d]^{\overline{\CT}_{\frac{m}{n} \ra \frac{m}{n}+1}} \overline{\CA}_{\frac{m}{n}}(n,r)\text{-}\on{mod} \\
D(\overline{R})\text{-}\on{mod}^{G,\frac{m}{n}+1} \ar[rr]^{\overline{\pi}_{\frac{m}{n}+1}} && \overline{\CA}_{\frac{m}{n}+1}(n,r)\text{-}\on{mod},
} 
 \end{equation}
 where we denote by $\overline{\pi}_{\frac{m}{n}},\overline{\pi}_{\frac{m}{n}+1}$ the Hamiltonian reduction functors. 
Let us now define the sheaf versions of the functors $\overline{\pi}_{\frac{m}{n}},\overline{\pi}_{\frac{m}{n}+1},\overline{\CT}_{c\ra c+1}$. We denote
$$\overline{\pi}_c^\theta \colon D_{\overline{R}}\text{-}\on{mod}^{G,c} \ra \overline{\CA}_c^\theta(n,r)\text{-}\on{mod}, \quad \CF \mapsto (\CF|_{(T^*\overline{R})^{\theta-\on{st}}})^{\on{GL}(V)},$$

\noindent the functor $\overline{\pi}^\theta_{c+1}$ can be defined similarly. We denote by $\overline{\CT}^\theta_{c \ra c+1}$ the functor from $\overline{\CA}_c^\theta(n,r)\text{-}\on{mod}$ to $\overline{\CA}_{c+1}^\theta(n,r)\text{-}\on{mod}$ given by left tensoring with the sheaf of bimodules 
$(\overline{\CQ}_c|_{(T^*\overline{R})^{\theta-\on{st}}})^{\on{GL}(V),\on{det}}$. 
It follows from~\cite[Proposition~6.31]{BPW} that the following diagram is commutative: 
\begin{equation*}
\xymatrix{
\overline{\CA}^\theta_c(n,r)\text{-}\on{mod} \ar[r]^{\overline{\Gamma}(\bullet)} \ar[d]^{\overline{\CT}^\theta_{c \ra c+1}} & \overline{\CA}_c(n,r)\text{-}\on{mod} \ar[d]^{\overline{\CT}_{c \ra c+1}} \\
\overline{\CA}^\theta_{c+1}(n,r)\text{-}\on{mod} \ar[r]^{\overline{\Gamma}(\bullet)} & \overline{\CA}_{c+1}(n,r)\text{-}\on{mod},
}
\end{equation*}
so to prove that~(\ref{trans_ham}) is commutative it remains to check the commutativity of the following diagram:
\begin{equation*}
\xymatrix{
D_{\overline{R}}\text{-}\on{mod}^{G,\frac{m}{n}} \ar[rr]^{\overline{\pi}^\theta_{\frac{m}{n}}} \ar[d]^{\widetilde{\CT}_{\frac{m}{n} \ra \frac{m}{n}+1}} && \overline{\CA}^\theta_c(n,r)\text{-}\on{mod} \ar[d]^{\overline{\CT}^\theta_{\frac{m}{n} \ra \frac{m}{n}+1}} \\
D_{\overline{R}}\text{-}\on{mod}^{G,\frac{m}{n}+1} \ar[rr]^{\overline{\pi}^\theta_{\frac{m}{n}+1}} && \overline{\CA}^\theta_{\frac{m}{n}+1}(n,r)\text{-}\on{mod}.
}
\end{equation*}
This was observed in~\cite[(5.1)]{bl}.
\epr

It follows from Lemma~\ref{trans_equiv_simp} that we can assume that $\frac{m}{n} \in \BC$ is Zariski generic, so we have an isomorphism $\mathsf{C}_{\nu_0}(\CA_{\frac{m}{n}}(n,r))\simeq \bigoplus\bigotimes_i \CA_{\frac{m}{n}+r-i}(n_i,1)$ (see Proposition~\ref{fixed_points}), where the sum is taken over all ordered collections $(n_{1}, \dots, n_{r})$ of non-negative integers such that $n_{1}+\ldots+n_{r}=n$.

According to~\cite[Sections~6.3,~6.3]{L} the module $L(\varnothing,\ldots,\varnothing,n_0\la')$ can be described as follows. Let $L^A(n_0\la')$ be the module in category $\mathcal{O}$ over 
$\CA_{\frac{m}{n}}(n,1)$ corresponding to $n_0\la'$. We have the projection
$\kappa\colon \mathsf{C}_{\nu_0}(\CA_{\frac{m}{n}}(n,r)) \twoheadrightarrow \CA_{\frac{m}{n}}(n,1)$ which makes $L^A(n_0\la')$ a module over $\mathsf{C}_{\nu_0}(\CA_{\frac{m}{n}}(n,r))$. Then 
$L(\varnothing,\ldots,\varnothing,n_0\la')$ is the maximal quotient of $\Delta_{\nu_0}(L^A(n_0\la'))$ that does not intersect the highest weight space $L^A(n_0\la')$, where
$$\Delta_{\nu_0}(L^A(n_0\la')):=\CA_{\frac{m}{n}}(n,r) \otimes_{\CA^{\geqslant 0,\nu_0}_{\frac{m}{n}}(n,r)} L^A(n_0\la').$$ 
\noindent We conclude that 
$L(\varnothing,\ldots,\varnothing,n_0\la')^{\textrm{hw}} \simeq L^A(n_0\la')$ as modules over $\mathsf{C}_{\nu_0}(\CA_{\frac{m}{n}}(n,r))$, 
here $L_{\frac{m}{n},r}(\varnothing,\ldots,\varnothing,n_0\la')^{\textrm{hw}}$ is the highest weight component of $L_{\frac{m}{n},r}(n_0\la)$ w.r.t. $\nu_0$. 
Note also that it follows from~\cite[Proposition~7.8]{egl} that we have an isomorphism between $\CA_{\frac{n}{m}}(n,1)$-modules $L_{\frac{n}{m},1}(n_0\la')$ and $L^A(n_0\la')$.
So to show that $\la'=\la$ it is enough to check that ${L_{\frac{m}{n},r}(n_0\la)^{\textrm{hw}} \simeq L_{\frac{m}{n},1}(n_0\la)}$ as a module over $\CA_{\frac{m}{n}}(n,1) \hookrightarrow \mathsf{C}_{\nu_0}(\CA_{\frac{m}{n}}(n,r))$.

Note that we have a natural isomorphism of vector spaces $L_{\frac{m}{n},r}(n_0\la)^{\textrm{hw}} \simeq L_{\frac{m}{n},1}(n_0\la)$. To see this let us denote by $W^{\textrm{lw}} \subset W$ the lowest weight component with respect to the $\BC^\times$-action via $\nu_0$.
Set $R^{\textrm{lw}}:=\mathfrak{gl}(V) \oplus \on{Hom}(V,W^{\textrm{lw}}) \subset R$, $M:=M(O(n_0\la))$. We have a tautological identification $\overline{L}_{\frac{m}{n},r}(n_0\la)^{\textrm{hw}}= (M\otimes \BC[\on{Hom}(V,W^{\textrm{lw}})])^{\on{GL}(V)}=\overline{L}_{\frac{m}{n},1}(n_0\la)$. It remains to prove the following proposition. 
		\prop{}\label{main_prop_sec}
		The homomorphism $\kappa\colon \mathsf{C}_{\nu_0}(\CA_{\frac{m}{n}}(n,r)) \twoheadrightarrow \CA_{\frac{m}{n}}(n,1)$ intertwines the actions 
		 \begin{equation}\label{interwine_main}
		 \mathsf{C}_{\nu_0}(\CA_{\frac{m}{n}}(n,r)) \curvearrowright L_{\frac{m}{n},r}(n_0\la)^{\textrm{hw}}=L_{\frac{m}{n},1}(n_0\la) \curvearrowleft \CA_{\frac{m}{n}}(n,1).
		 \end{equation}
		\eprop
		 
		 We will prove this Proposition in the end of this section. Recall that the cocharacter $\nu_0\colon \BC^\times \ra T$ is given by $t \mapsto (1,t^{d_1},\ldots,t^{d_r})$ with $d_1 \gg \ldots \gg d_r$. 
		 Consider now the following cocharacter $\nu_0'\colon \BC^\times \ra T,\, t \mapsto (1,t^{d_1-d_r},t^{d_2-d_r}\ldots,t^{d_{r-1}-d_r},1)$. Note that the $\BC^\times$-actions on $\mathfrak{M}(n,r),\, \mathfrak{M}^\theta(n,r)$ corresponding to $\nu_0, \nu_0'$ coincide. So we have the identification
		 $\mathsf{C}_{\nu_0'}(\CA_{\frac{m}{n}}(n,r))=\mathsf{C}_{\nu_0}(\CA_{\frac{m}{n}}(n,r))$. Note also that $R^{\textrm{lw}}=R^{\nu'_0}$. 
		 
		 Let us now note that we have the natural isomorphism $\mathsf{C}_{\nu_0'}(D(R)) \simeq D(R^{\textrm{lw}})$ induced by the embedding $D(R^{\textrm{lw}}) \hookrightarrow D(R)$ which clearly intertwines the actions
		 \begin{equation}\label{inter_D}
		  \mathsf{C}_{\nu_0'}(D(R)) \curvearrowright (M \otimes \BC[\on{Hom}(V,W)])^{\textrm{hw}}=M\otimes \BC[\on{Hom}(V,W^{\textrm{lw}})] \curvearrowleft  D(R^{\textrm{lw}}).
		 \end{equation}

Set $c:=\frac{m}{n}$. Let us now understand the relation between the homomorphism 
\begin{equation*} \kappa\colon \mathsf{C}_{\nu'_0}(D(R)/\!\!/\!\!/_{\frac{m}{n}}\on{GL}(V)) \twoheadrightarrow D(R^{\textrm{lw}})/\!\!/\!\!/_{\frac{m}{n}}\on{GL}(V)
\end{equation*} and the isomorphism $\mathsf{C}_{\nu'_0}(D(R)) \simeq D(R^{\textrm{lw}})$ above. 
Note that both of them are induced by the corresponding isomorphisms of sheaves 
\begin{equation*} \mathsf{C}_{\nu'_0}(\CA^\theta_{\frac{m}{n}}(n,r))|_{\mathfrak{M}^\theta(n,1)} \overset{\kappa^\theta}{\simeq} \CA^\theta_{\frac{m}{n}}(n,1),\, 
\mathsf{C}_{\nu'_0}(D_R) \simeq D_{R^{\textrm{lw}}}.
\end{equation*}

We start with two general lemmas. Let $\CA$ be an associative algebra equipped with a $\BZ$-grading $\CA=\bigoplus_{i \in \BZ}\CA^i$ and $J \subset \CA$ is a $\BZ$-graded two-sided ideal. 

\lem{}\label{cartan_via_quot}		 
We have a natural 
isomorphism $\mathsf{C}(\CA)/[J^{\geqslant 0}] \iso \mathsf{C}(\CA/J)$, where $\mathsf{C}$ corresponds to taking the Cartan subquotient and $[J^{\geqslant 0}] \subset \mathsf{C}(\CA)=\CA^{\geqslant 0}/(\CA^{\geqslant 0} \cap \CA\CA^{>0})$ is the image of $J^{\geqslant 0}$ under the natural morphism $J^{\geqslant 0} \ra \mathsf{C}(\CA)$.
\elem		 
\prf
Both of them can be naturally identified with $\CA^{\geqslant 0}/((\CA^{\geqslant 0} \cap \CA\CA^{>0})+J^{\geqslant 0})$.
\epr

Let $\CA$ be an associative algebra as above and assume that the $\BZ$-grading is
induced by some element $h \in \CA$, i.e. $\CA^i=\{a \in \CA\,|\, [h,a]=ia\}$.
Let $\mathfrak{l}$ be a reductive Lie algebra. Assume that we are given a locally-finite completely reducible action $\mathfrak{l} \curvearrowright \CA$ which commutes with the $\BZ$-grading and is induced by a map of Lie algebras $\phi\colon \mathfrak{l} \ra \on{Der}(\CA)$. Assume also that we have a quantum comoment map $\upsilon\colon \CU(\mathfrak{l}) \ra \CA$ for this action, i.e. a homomorphism of algebras $\upsilon$ such that $[\upsilon(x),-]=\phi(x)(-),\, \forall x \in \mathfrak{l}$.
\rem{}\label{image_comom}
Note that the image of $\upsilon$ lies in $\CA^0$.
Indeed, pick $x \in \mathfrak{l}$, we have to show that $\upsilon(x) \in \CA^0$. Note that $h \in \CA^0$ and $\phi(x)(\CA^0) \subset \CA^0$ so we must have $[h,\upsilon(x)] \in \CA_0$. We can decompose $\upsilon(x)=\sum_{i \in \BZ} a_i$ with $a_i \in \CA_i$ and note that $[h,\upsilon(x)]=\sum_{i \in \BZ} ia_i$ lies in $\CA^0$ only if $a_i=0$ for every $i \neq 0$. The claim follows. 
\erem

Fix a character $\la\colon \mathfrak{l} \ra \BC$. Let $P \subset \CA$ be the left ideal generated by ${\{\xi(x)-\la(x)\,|\, x \in \mathfrak{l}\}}$. We define $\CA/\!\!/\!\!/_\la \mathfrak{l}:=(\CA/P)^{\mathfrak{l}}=\CA^{\mathfrak{l}}/P^{\mathfrak{l}}$. We analogously define $\mathsf{C}(\CA)/\!\!/\!\!/_\la \mathfrak{l}$ using the quantum comoment map $[\upsilon]\colon \CU(\mathfrak{l}) \ra \mathsf{C}(\CA)$ which is well defined by Remark~\ref{image_comom}. Note that $P^{\mathfrak{l}} \subset \CA^{\mathfrak{l}}$ is a two-sided ideal.

\lem{}\label{main_morphism}
For any character $\la\colon \mathfrak{l} \ra \BC$ we have a natural epimorphism 
$$\mathsf{C}(\CA/\!\!/\!\!/_\la \mathfrak{l}) \twoheadrightarrow \mathsf{C}(\CA)/\!\!/\!\!/_\la \mathfrak{l}.$$
\elem
\prf
By Lemma~\ref{cartan_via_quot} (applied to $\CA^\mathfrak{l} \supset P^{\mathfrak{l}}$) we have ${\mathsf{C}(\CA/\!\!/\!\!/_\la \mathfrak{l})=\mathsf{C}(\CA^{\mathfrak{l}})/[(P^{\geqslant 0})^{\mathfrak{l}}]}$. Note also that by the definitions $\mathsf{C}(\CA)/\!\!/\!\!/_\la \mathfrak{l}=\mathsf{C}(\CA)^{\mathfrak{l}}/[(P^{\geqslant 0})^{\mathfrak{l}}]$. Now the claim follows from the fact that $\on{Id}\colon \CA \ra \CA$ induces a surjective homomorphism $\mathsf{C}(\CA^{\mathfrak{l}}) \twoheadrightarrow
\mathsf{C}(\CA)^{\mathfrak{l}}$. 
\epr

Using Lemma~\ref{main_morphism} and the fact that the open sets $((T^*R)_f/\!\!/\!\!/^\theta\on{GL}(V))^{\nu_0'(\BC^\times)}$, 
where $f\in \BC[T^*R]^{\nu'_0(\BC^\times)}\cap 
\BC[T^*R]^{G,\theta}$ is homogeneous 
of positive degree, form a basis for the conical Zariski topology on $\mathfrak{M}^\theta(n,r)^{\nu_0'(\BC^\times)}$, we obtain a homomorphism 
$\mathsf{C}_{\nu'_0}(D_R/\!\!/\!\!/^\theta_c \on{GL}(V))|_{\mathfrak{M}^\theta(n,r)} \ra \mathsf{C}_{\nu'_0}(D_R)/\!\!/\!\!/^\theta_c \on{GL}(V)$ of sheaves on $T^*R^{\textrm{lw}}/\!\!/\!\!/^\theta\on{GL}(V)$.
\lem{}\label{sheaf_main_morphism}
The homomorphism $\mathsf{C}_{\nu'_0}(D_R/\!\!/\!\!/^\theta_c \on{GL}(V))|_{\mathfrak{M}^\theta(n,1)} \ra \mathsf{C}_{\nu'_0}(D_R)/\!\!/\!\!/^\theta_c \on{GL}(V)$ is an isomorphism.
\elem
\prf
Note that $\mathsf{C}_{\nu'_0}(D_R/\!\!/\!\!/^\theta_c \on{GL}(V))|_{\mathfrak{M}^\theta(n,1)},\, \mathsf{C}_{\nu'_0}(D_R)/\!\!/\!\!/^\theta_c \on{GL}(V)$ are filtered and by~\cite[Proposition~5.2 (2)]{Cat_O_quant} we have 
\begin{equation*}
\on{gr}\mathsf{C}_{\nu'_0}(D_R/\!\!/\!\!/^\theta_c G)|_{\mathfrak{M}^\theta(n,1)}=\CO_{\mathfrak{M}^\theta(n,1)}=\on{gr}\mathsf{C}_{\nu'_0}(D_R)/\!\!/\!\!/^\theta_c \on{GL}(V).
\end{equation*}
The homomorphism $\mathsf{C}_{\nu'_0}(D_R/\!\!/\!\!/^\theta_c \on{GL}(V))|_{\mathfrak{M}^\theta(n,1)} \ra \mathsf{C}_{\nu'_0}(D_R)/\!\!/\!\!/^\theta_c \on{GL}(V)$ preserves filtrations and the associated graded equals to $\on{Id}\colon \CO_{\mathfrak{M}^\theta(n,1)} \iso \CO_{\mathfrak{M}^\theta(n,1)}$. The claim follows.
\epr

Lemma~\ref{sheaf_main_morphism} gives us an explicit construction of an isomorphism
\begin{equation*}
\mathsf{C}_{\nu'_0}(D_R/\!\!/\!\!/^\theta_c \on{GL}(V))|_{\mathfrak{M}^\theta(n,1)} \iso \mathsf{C}_{\nu'_0}(D_R)/\!\!/\!\!/^\theta_c \on{GL}(V).
\end{equation*}
Note that the sheaf $\mathsf{C}_{\nu'_0}(D_R/\!\!/\!\!/^\theta_c \on{GL}(V))|_{\mathfrak{M}^\theta(n,1)}$ is exactly $\mathsf{C}_{\nu_0'}(\CA^\theta_c(n,r))$.
We claim that the sheaves $\CA^\theta_{c}(n,1),\, \mathsf{C}_{\nu_0'}(D_R)/\!\!/\!\!/^\theta_c \on{GL}(V)$ are canonically isomorphic and the isomorphism between them is induced by the isomorphism $D(R^{\textrm{lw}}) \iso \mathsf{C}_{\nu_0'}(D(R))$.
To see that it is enough to prove the following lemma.
\lem{}\label{pass}
The isomorphism $D(R^{\textrm{lw}}) \iso \mathsf{C}_{\nu'_0}(D(R))$ induces
an isomorphism 
$I_{\textrm{lw}} \iso [I^{\geqslant 0,\nu'_0}]$, where
$$
I=D(R)\{\xi_R-c\on{tr}\xi\,|\,\xi \in \mathfrak{gl}(V)\}, \, \text{and}\, I_{\textrm{lw}}=D(R^{\textrm{lw}})\{\xi_{R^{\textrm{lw}}}-c\on{tr}\xi\,|\,\xi \in \mathfrak{gl}(V)\}.
$$
\elem
\prf
It is enough to check that that the isomorphism $D(R^{\textrm{lw}}) \iso \mathsf{C}_{\nu'_0}(D(R))$ induces a surjective map $I_{\textrm{lw}} \twoheadrightarrow [I^{\geqslant 0,\nu'_0}]$. The ideal $I_{\textrm{lw}}$ is generated by the elements of the form $\xi_{R^{\textrm{lw}}}-c\on{tr}\xi,\, \xi \in \mathfrak{gl}(V)$. Note that $\xi_R-\xi_{R^{\textrm{lw}}} \in D(R)^{>0,\nu_0'}$ because the action of $\nu_0'(\BC^\times)$ contracts $R$ to $R^{\textrm{lw}}=R^{\nu_0'(\BC^\times)}$. It then follows that, in the Cartan subquotient $\mathsf{C}_{\nu_0'}(D(R))\simeq D(R)^{\geqslant 0,\nu_0'}/(D(R)^{\geqslant 0,\nu_0'} \cap D(R)D(R)^{>0,\nu_0'})$ we have ${[\xi_{R^{\textrm{lw}}}-c\on{tr}\xi]=[\xi_R-c\on{tr}\xi]}$. The claim follows.

\epr

\begin{proof}[Proof of Proposition~\ref{main_prop_sec}] 
Combining Lemmas~\ref{main_morphism},~\ref{sheaf_main_morphism},~\ref{pass} and using the constructions therein we see that the isomorphism
\begin{equation*}
\kappa^\theta\colon \mathsf{C}_{\nu'_0}(\CA^\theta_{\frac{m}{n}}(n,r))|_{\mathfrak{M}^\theta(n,r)} \simeq \CA^\theta_{\frac{m}{n}}(n,1)
\end{equation*}
is induced by the natural embedding $D(R^{\textrm{lw}}) \hookrightarrow D(R)$.
Now the desired statement about the intertwining property of $\kappa\colon \mathsf{C}_{\nu_0}(\CA_{\frac{m}{n}}(n,r)) \twoheadrightarrow \CA_{\frac{m}{n}}(n,1)$ follows from the fact that the isomorphism $\mathsf{C}_{\nu_0'}(D(R)) \simeq D(R^{\textrm{lw}})$ intertwines the actions in~(\ref{inter_D}). 
\end{proof}

\section{Character of $L_{\frac{m}{n},r}(n_0\la)=L_{\nu}(\varnothing,\ldots,\varnothing,n_{0}\lambda)$}\label{char_min_supp}	
Our construction of $L_{\frac{m}{n},r}(n_0\la)$ allows us to compute its character. Recall that we have set $m_{0} := m/\gcd(m,n)$ and $n_{0}:=n/\gcd(m,n)$. 
\ssec{}{Characters of minimally supported modules over $H_{\frac{n}{m}}(m,1)$}
Recall that $S_{\frac{n}{m}}(m_0\la)$ is the irreducible highest weight module over $H_{\frac{n}{m}}(m,1)$ with highest weight $m_0\la$.
The character of $S_{\frac{n}{m}}(m_0\la)$ was computed in~\cite[Theorem~1.4]{egl}. Let us recall the answer.
Let $\La$ be the ring of symmetric functions on infinitely many variables $z_1,z_2,\ldots$. For a partition $\beta$ of $m$ we define a constant $c^\beta_{\la,m_0}$ by
		 \begin{equation*}
	    s_{\la}(z_1^{m_0},z_2^{m_0},\ldots)=\sum_\beta c^\beta_{\la,m_0}s_\beta(z_1,z_2,\ldots),
		 \end{equation*}
		 where $s_{\la},\, s_\beta \in \La$ are the corresponding Schur polynomials. 	
\prop{}\label{S_via_stand}		 
The class $[S_{\frac{n}{m}}(m_0\la)] \in K_0(\CO(H_{\frac{n}{m}}(m,1)))$ is given by the formula
\begin{equation*}
[S_{\frac{n}{m}}(m_0\la)]=\sum_{\beta\vdash m}c^{\beta}_{\la,m_0}[\Delta_{\frac{n}{m}}(\beta)],
\end{equation*}
where $\Delta_{\frac{n}{m}}(\beta)$ is the standard object with highest weight $\beta$ in the category $\CO(H_{\frac{n}{m}}(m,1))$.
\eprop
\begin{proof}
The module $\Delta_{\frac{n}{m}}(\beta)$ is the graded dual of the co-standard module ${\nabla_{\frac{n}{m}}(\beta) \in \mathcal{O}(H_{\frac{n}{m}}(m, 1), \mathfrak{h})}$ of modules with locally nilpotent action of $\mathfrak{h}$. In the category $\mathcal{O}$ for $H_{\frac{n}{m}}(m, 1)$, the classes in $K_0$ of standard and co-standard modules coincide. Since taking the graded dual preserves the labels in category $\mathcal{O}$, the result now follows from \cite[Theorem 1.4]{egl}.
\end{proof}

Recall that for a finite dimensional representation $V$ of $S_m$ its Frobenius character is
\begin{equation}\label{frob_char}
\on{ch}_{S_m}V:=\frac{1}{m!}\sum_{\sigma \in S_m}\on{Tr}_V(\sigma)p_1^{k_1(\sigma)}\ldots p_l^{k_l(\sigma)} \in \La,
\end{equation}
here $p_i \in \La$ are power sums, $k_i(\sigma)$ is the number of cycles of length $i$ in $\sigma$, and $\La$ is the algebra of symmetric functions. 
For a partition $\beta$ of $m$
the Frobenius character of the irreducible representation $V_\beta$ is given by the Schur polynomial $s_\beta \in \La$. We will use plethystic notation, so that $f\left[\frac{X}{1 - q}\right]$ denotes the image of $f \in \La$ under the automorphism that sends power sums $p_k$ to $p_{k}\left[\frac{X}{1 - q}\right] = \frac{p_{k}}{1 - q^{k}}$.

\lem{}\label{S_m_char_stand}
For a partition $\beta$ of $m$ we have
\begin{equation*}
\on{ch}_{q,S_m}(\Delta_{\frac{n}{m}}(\beta))=(1-q^{-1})q^{-\frac{m-1}{2}+\frac{n}{m}\kappa(\beta)}s_\beta\left[\frac{X}{1-q^{-1}}\right],
\end{equation*}
where $\kappa(\beta)$ is the sum of contents of all boxes of $\beta$.
\elem
\prf
It follows from~\cite{beg2} that the highest weight component of $\Delta_{\frac{n}{m}}(\beta)$ has weight $q^{\frac{n}{m}\kappa(\beta)-\frac{m-1}{2}}$.
The module $\Delta_{\frac{n}{m}}(\beta)$ is isomorphic to $V_\beta \otimes \BC[\mathfrak{h}]$ as $S_m \times \BC^\times$-module and the $\BC^\times$-action corresponds to the shifted standard negative grading $\BC[\mathfrak{h}]=\bigoplus_{k \geqslant 0}S^k(\mathfrak{h}^*)$, $\on{deg}(S^k(\mathfrak{h}^*))=-k-\frac{m-1}{2}+\frac{n}{m}\kappa(\beta)$.
Consider now a permutation $\sigma \in S_m$. It is clear that 
$\on{det}_{\mathfrak{h}}(1-q^{-1}\sigma)=\frac{1}{1-q^{-1}}\prod_i (1-q^{-i})^{k_i(\sigma)}$. Note also that 
$$\on{Tr}_{V_\nu \otimes \BC[\mathfrak{h}]}(\sigma q^{h})=\frac{\on{Tr}_{V_\nu}(\sigma)}{\on{det}_{\mathfrak{h}}(1-q^{-1}\sigma)}=(1-q^{-1})\frac{\on{Tr}_{V_\beta}(\sigma)}{\prod_i (1-q^{-i})^{k_i(\sigma)}}.$$

\noindent We conclude that 
\begin{multline*}
\on{ch}_{q,S_m}(\Delta_{\frac{n}{m}}(\beta))=\frac{1}{m!}\sum_{\sigma \in S_m}(1-q^{-1})q^{\frac{n}{m}\kappa(\beta)-\frac{m-1}{2}}\frac{\on{Tr}_{V_\beta}(\sigma)\prod_i p_i^{k_i(\sigma)}}{\prod_i (1-q^{-i})^{k_i(\sigma)}}=\\
=(1-q^{-1})q^{\frac{n}{m}\kappa(\beta)-\frac{m-1}{2}}s_{\beta}\left[\frac{X}{1 - q^{-1}}\right].
\end{multline*}
\epr

\cor{}\label{char_S}
The $q$-graded $S_m$-character of $S_{\frac{n}{m}}(m_0\la)$ is given by 
\begin{equation*}
\on{ch}_{q,S_m}(S_{\frac{n}{m}}(m_0\la))=(1-q^{-1})\sum_{\beta\vdash m}c^{\beta}_{\la,m_0}q^{-\frac{m-1}{2}+\frac{n}{m}\kappa(\beta)}s_\beta\left[\frac{X}{1-q^{-1}}\right].
\end{equation*}
\ecor
\prf
Follows from Proposition~\ref{S_via_stand} and Lemma~\ref{S_m_char_stand}.
\epr

\sssec{}{Computation of the character of $L_{\frac{m}{n},r}(n_0\la)$}
Let us now finally compute the $q$-graded $\on{GL}_r$ character of the module $L_{\frac{m}{n},r}(n_0\la)$. Recall that we have a $q$-graded $\on{GL}_r$-equivariant isomorphism 
\begin{equation}\label{main_iso}
L_{\frac{m}{n},r}(n_0\la) \iso (S_{\frac{n}{m}}(m_0\la) \otimes (\BC^{r*})^{\otimes m})^{S_m}.   
\end{equation}
\prop{}
We have 
\begin{multline*}
\on{ch}_{q,\on{GL}_r}(L_{\frac{m}{n},r}(n_0\la))=\\=
(1-q^{-1})\sum\limits_{\substack{\on{r}(\mu)\leqslant \on{min}(n,r)\\\mu,\beta\vdash m}}c^{\beta}_{\la,m_0}q^{-\frac{m-1}{2}+\frac{n}{m}\kappa(\beta)}\langle s_\beta\left[\frac{X}{1-q^{-1}}\right],s_{\mu}\rangle [W_r(\mu)^{*}],
\end{multline*}
where $\langle\,,\,\rangle$ is the Hall inner product on $\La$, i.e.  the inner product with respect to which $\langle s_\al,s_\gamma \rangle=\delta_{\al\gamma}$ for any two partitions $\al,\,\gamma$.
\eprop
\prf
By Corollary~\ref{char_S} we have 
\begin{equation*}
\on{ch}_{q,S_m}(S_{\frac{n}{m}}(m_0\la^t))=(1-q^{-1})\sum_{\beta,|\beta|=m}c^{\beta}_{\la,m_0}q^{-\frac{m-1}{2}+\frac{n}{m}\kappa(\beta)}s_\beta\left[\frac{X}{1-q^{-1}}\right].
\end{equation*}
By Schur-Weyl duality we have
\begin{equation*}
\on{ch}_{S_m \times \on{GL}_r}((\BC^{r*})^{\otimes m})=\sum\limits_{\substack{\on{r}(\beta)\leqslant \on{min}(n,r)\\|\beta|=m}}s_\beta [W_r(\beta)^{*}].
\end{equation*}
So from~(\ref{main_iso}) we obtain the desired equality.  
\epr

\end{document}